\documentclass{amsart}
\usepackage{amsmath,amsthm,amsfonts,amssymb,latexsym,mathrsfs,graphicx}
\usepackage{yfonts}

\usepackage{hyperref}
\usepackage{multicol}
\usepackage[table]{xcolor}
\usepackage{color, soul}
\sethlcolor{green}

\usepackage{comment}

\usepackage{enumerate}
\usepackage[shortlabels]{enumitem}

\newcommand{\Z}{\mathbb Z}
\newcommand{\Q}{\mathbb Q}
\newcommand{\C}{\mathbb C}
\newcommand{\R}{\mathbb R}

\newcommand{\normal}{\unlhd}

\newcommand{\rg}{\mathcal{R}}
\newcommand{\sg}{\mathcal{S}}

\newcommand{\Irr}{\text{Irr}}
\newcommand{\Syl}{\text{Syl}}
\newcommand{\Gal}{\textnormal{Gal}}

\newtheorem{theorem}{Theorem}[section]

\newtheorem{lemma}[theorem]{Lemma}
\newtheorem{proposition}[theorem]{Proposition}

\newtheorem*{ThmA}{Theorem A}
\newtheorem*{Main}{Main Theorem}

\theoremstyle{definition}
\newtheorem{definition}[theorem]{Definition}

\newtheorem{rem}[theorem]{Remark}

\def\irr#1{{\rm Irr}(#1)}
\def\cent#1#2{{\bf C}_{#1}(#2)}

\def\bg#1#2{{\bf B}_{#1}(#2)}

\def\syl#1#2{{\rm Syl}_#1(#2)}

\def\oh#1#2{{\bf O}_{#1}(#2)}
\def\zent#1{{\bf Z}(#1)}

\def\aut#1{{\rm Aut}(#1)}

\newcommand{\F}{{\mathbb F}}

\def\irr#1{{\rm Irr}(#1)}

\def\cent#1#2{{\bf C}_{#1}(#2)}

\def\syl#1#2{{\rm Syl}_#1(#2)}

\def\norm#1#2{{\bf N}_{#1}(#2)}
\def\oh#1#2{{\bf O}_{#1}(#2)}

\def\zent#1{{\bf Z}(#1)}

\def\aut#1{{\rm Aut}(#1)}

\def\Z{{\mathbb Z}}
\def\C{{\mathbb C}}
\def\Q{{\mathbb Q}}
\def\irr#1{{\rm Irr}(#1)}

\def\cent#1#2{{\bf C}_{#1}(#2)}
\def\syl#1#2{{\rm Syl}_#1(#2)}
\def\oh#1#2{{\bf O}_{#1}(#2)}

\def\zent#1{{\bf Z}(#1)}

\def\norm#1#2{{\bf N}_{#1}(#2)}

    \def \mod#1{\, {\rm mod} \, #1 \, }

\mathchardef\coso="2023

\begin{document}

\title{On quadratic rational Frobenius 
 groups}

\author[E. Pacifici]{Emanuele Pacifici}
\address{Emanuele Pacifici, Dipartimento di Matematica e Informatica U. Dini,\newline
Universit\`a degli Studi di Firenze, viale Morgagni 67/a,
50134 Firenze, Italy.}
\email{emanuele.pacifici@unifi.it}

\author[M. Vergani]{Marco Vergani}
\address{Marco Vergani, Dipartimento di Matematica e Informatica U. Dini,\newline
Universit\`a degli Studi di Firenze, viale Morgagni 67/a,
50134 Firenze, Italy.}
\email{marco.vergani@unifi.it}

\thanks{The authors are partially supported by INdAM-GNSAGA. This research is also partially funded by the European Union-Next Generation EU, Missione 4 Componente 1, CUP B53D23009410006, PRIN 2022 2022PSTWLB - Group Theory and Applications.}


\begin{abstract} 
Let $G$ be a finite group and, for a given complex character $\chi$ of $G$, let $\Q(\chi)$ denote the field extension of $\Q$ obtained by adjoining all the values $\chi(g)$, for $g\in G$.
The group $G$ is called \emph{quadratic rational} if, for every irreducible complex character $\chi\in\irr G$, the field $\Q(\chi)$ is an extension of $\Q$ of degree at most $2$. Quadratic rational groups have a nice characterization in terms of the structure of the group of central units in their integral group ring, and in fact they generalize the well-known concept of a \emph{cut} group (i.e., a finite group whose integral group ring has a finite group of central units). In this paper we classify the Frobenius groups that are quadratic rational, a crucial step in the project of describing the Gruenberg-Kegel graphs associated to quadratic rational groups. It turns out that every quadratic rational Frobenius group is \emph{uniformly semi-rational}, i.e., it satisfies the following property: all the generators of any cyclic subgroup of $G$ lie in at most two conjugacy classes of $G$, and these classes are permuted by the same element of the Galois group  ${\rm{Gal}}(\Q_{|G|}/\Q)$ (in general, every cut group is uniformly semi-rational, and every uniformly semi-rational group is quadratic rational). We will also see that the class of groups here considered coincides with the one studied in \cite{semiratfrobenius}, thus the main result of this paper also completes the analysis carried out in \cite{semiratfrobenius}.
\end{abstract}

\maketitle

\section{Introduction}

A finite group $G$ is called \emph{quadratic rational} if, for every irreducible complex character $\chi\in\irr G$, the field of values $\Q(\chi)=\Q(\chi(g)\mid g\in G)$ is a field extension of $\Q$ of degree at most $2$. A closely related concept is that of a \emph{semi-rational group}: a finite group $G$ is called semi-rational if, for every element $g\in G$, the generators of $\langle g\rangle$ lie either in the conjugacy class of $g$ or in that of $g^{r_g}$ for a suitable integer $r_g$ (depending on $g$ and coprime to its order). This is equivalent to the property that, for every $g\in G$, the field $\Q(g)=\Q(\chi(g)\mid\chi\in\irr G)$ is an extension of $\Q$ of degree at most $2$; thus, from a character table perspective, the quadratic rationality and semi-rationality properties are dual to each other. 

Quadratic rational groups have been studied in \cite{tentquadraticrat}; in particular, it has been proved that the set $\pi(G)$ of primes that divide the order of a solvable quadratic rational group $G$ is contained in $\{2,3,5,7,13\}$ (and, for every $p\in\{2,3,5,7,13\}$, there exists a solvable quadratic rational group whose order is divisible by $p$). A similar result has been obtained also for solvable semi-rational groups in \cite{CD}: the set $\pi(G)$ of any solvable semi-rational group $G$ is contained in $\{2,3,5,7,13,17\}$, but in this case, at the time of this writing, it is not known whether there exists a solvable semi-rational group $G$ such that $17\in\pi(G)$. Both the classes of quadratic rational groups and semi-rational groups contain the so called \emph{cut} groups, that have been extensively studied because of their characterization in terms of the central units of their integral group ring: a  group $G$ is called a cut group if all the central units of its integral group ring $\Z G$ are trivial (i.e., of the form $\pm g$ for some $g\in \zent G$). It has been proved in Chapter~8 of \cite{Bovdi} and in \cite{rs} that this ring-theoretical condition is equivalent to the fact that, for every $\chi\in\irr G$, the field $\Q(\chi)$ is either $\Q$ or a purely imaginary extension of $\Q$ of degree $2$. Hence, a cut group $G$ is quadratic rational, and if $G$ is solvable then $\pi(G)\subseteq\{2,3,5,7\}$ (see~\cite{bac}). On the other hand, the cut property is also equivalent to \emph{inverse semi-rationality}: a group is cut if and only if it is semi-rational with $r_g=-1$ for every $g\in G$. Note that the class of cut groups is in turn a generalization of the important class of \emph{rational groups}, i.e., groups whose character table contains only rational entries (equivalently, groups in which the generators of every given cyclic subgroup lie in the same conjugacy class). We also note that quadratic rationality, as well as the cut property, is equivalent to one particular  condition concerning central units of the group ring (see Remark~\ref{RingCharacterization}). 

An interesting way to understand the properties of these classes of groups is to study their \emph{Gruenberg-Kegel graphs}. The Gruenberg-Kegel graph associated to a finite group $G$, also called \emph{prime graph} of $G$ and denoted by ${\rm{GK}}(G)$,  is the simple undirected graph whose vertices are the primes in $\pi(G)$, and two vertices $p$, $q$ are adjacent if and only if $G$ has an element of order $pq$. In \cite{gkdelrio}, graphs that are realizable as ${\rm{GK}}(G)$ for a solvable cut group $G$ are studied, and a crucial step in this analysis is to get control of the Frobenius and $2$-Frobenius cut groups, since they precisely constitute the cases in which the prime graph is disconnected.

Frobenius groups that are rational, inverse semi-rational or semi-rational have been studied in  \cite{qfrob}, \cite{bac} and \cite{semiratfrobenius} respectively. Our objective is to describe Frobenius quadratic rational groups, aiming to understand (in future research) the graphs that are realizable as Gruenberg-Kegel graphs of quadratic rational groups. To this end, it is relevant to introduce a class of groups that lies both in the class of semi-rational and quadratic rational groups: we say that a finite group is \emph{uniformly semi-rational} if it is semi-rational and there exists an integer $r$ such that $r_g=r$ for all $g\in G$ (see also~\cite{dRV}). When it is convenient to emphasize the role of the integer $r$, we also call such a group \emph{$r$-semi-rational}; thus, for instance, the cut groups are precisely the $-1$-semi-rational groups. Not surprisingly, $r$-semi-rational groups behave similarly to cut groups when $r^2$ is congruent to $1$ modulo the exponent of the group. 

In this paper we classify all the quadratic rational Frobenius groups, and we prove that they are all uniformly semi-rational. In contrast to the case of Frobenius cut groups, not all of them are solvable (in fact, there exists a unique non-solvable quadratic rational Frobenius group). 

Setting $H_1=\langle x,y\mid x^5=y^4=1, x^y=x^{-1} \rangle$, 
$$H_2=\langle x,y,z\mid x^3=y^8=[x,z]=1, z^2=y^4,  y^z=y^{-1}, x^y=x^{-1}\rangle,$$ which are groups of order $20$ and $48$ respectively, and denoting by $C_n$ a cyclic group of order~$n$, we obtain the following (in Theorem~A, as well as in the Main Theorem, we use the symbol ``$\rtimes$" to denote a semidirect product that is not direct). 

\begin{ThmA}
    Let $G$ be a quadratic rational Frobenius group with Frobenius complement $H$.
    \begin{enumerate}
        \item If $G$ is solvable, then $H$ is isomorphic to one of the groups in the following list: 
           $$C_2,\; C_3,\;C_4,\;C_6,\; Q_8,\; C_3\rtimes C_4,\; Q_{16},\; H_1, \;{\rm{SL}}_2(3),\; C_3\rtimes Q_8,\; C_3\times Q_8\;, {\rm{SL}}_2(3).C_2,\; H_2\,.$$
    \item If $G$ is non-solvable, then $H$ is isomorphic to ${\rm{SL}}_2(5)$.
    \end{enumerate}
Conversely,  every group $H$ as in (1) or (2) is the Frobenius complement of some quadratic rational Frobenius group.  
\end{ThmA}

In view of the above result, we are then able to prove the main theorem of this paper. It will be useful to introduce the following notation: if $G=K\rtimes H$ is a Frobenius group with Frobenius kernel $K$ and Frobenius complement $H$, then we set $G_n=K^n\rtimes H$ to be the Frobenius group given by the diagonal action of $H$ on the direct product of $n$ copies of $K$. Some of the groups in the following classification (namely, those in (1)(a), (1)(b) and (2)) are rational or cut groups, so they already appear in \cite{qfrob} and \cite{bac}; however, for the convenience of the reader, we present a complete statement. We take the occasion to point out that (2)(b) of the Main Theorem is slightly different from the original statement in \cite[Theorem~1.3, (2)(b)]{bac}, because the latter contains a (minor) inaccuracy: namely, it is not true in general that every cyclic subgroup of $K$ is normalized by $H$ (as shown by \cite[Example~1]{CD}).

\begin{Main}
    Let $G$ be a Frobenius group with Frobenius kernel $K$ and Frobenius complement~$H$. If $G$ is quadratic rational, then $G$ is uniformly semi-rational and the following conclusions hold.
    \begin{enumerate}
    \item Let $|H|$ be even. 
    \begin{enumerate}
        \item If $G$ is rational, then $G$ is isomorphic to one of the following:
         \begin{multicols}{2}
        \begin{enumerate}
            \item $C_3^n\rtimes C_2$.
            \item $C_3^{2n}\rtimes Q_8$.
            \item $C_5^2\rtimes Q_8$.
        \end{enumerate}
        \end{multicols}
        \item If $G$ is cut and not rational, then $G$ is isomorphic to one of the following:
        \begin{multicols}{2}
            \begin{enumerate}
                \item $C^{2n}_3\rtimes C_4$ 
                \item $C_5^n \rtimes C_4$ 
                \item $C_7^n \rtimes C_6$ 
                \item $C_5^2\rtimes (C_3\rtimes C_4)$ 
                \item $C_5^2 \rtimes {\rm {SL}}_2(3)$ 
                \item $C_7^2\rtimes {\rm {SL}}_2(3)$ 
                \item $C_7^{2n} \rtimes (C_3\times Q_8)$ 
            \end{enumerate}
        \end{multicols}
        \item If $G$ is $r$-semi-rational with $r^2\equiv 1\;(\mod\exp(G))$ and not cut, then $G$ is isomorphic to one of the following:
         \begin{multicols}{2}
        \begin{enumerate}
            \item $C_3^{4n} \rtimes Q_{16}$ 
            \item $C_7^2\rtimes ({\rm {SL}}_2(3).C_2)$ 
    \end{enumerate}
       \end{multicols}
        \item If $G$ is not as in (a), (b) or (c), then $G$ is $r$-semi-rational with $r^4\equiv 1\;(\mod\exp(G))$ and $G$ is isomorphic to one of the following:
        \begin{multicols}{2}
        \begin{enumerate}    
            \item  $C_5^n\rtimes C_2$
             \item $(C_5^a\times C_5^b)\rtimes C_4$ 
              \item $C_5^{2n}\rtimes C_6$ 
            \item $C_{13}^n\rtimes C_6$ 
            \item $C_5^{2n}\rtimes Q_8\quad(n>1)$ 
            \item $C_5^{2n}\rtimes (C_3\rtimes C_4)\quad(n>1)$ 
            \item $C_5^{4n}\rtimes Q_{16}$ 
            \item $C_3^{4n}\rtimes H_1$ 
            \item $C_5^{2n}\rtimes {\rm {SL}}_2(3)\quad(n>1)$ 
            \item $C_5^{4n}\rtimes (C_3\rtimes Q_8)$   
            \item $C_5^{4n}\rtimes (C_3\times Q_{8})$ 
            \item $C_{13}^{2n}\rtimes (C_3\times Q_{8})$ 
            \item $C_5^{4n}\rtimes H_2$ 
            \item $C_5^{4n}\rtimes ({\rm {SL}}_2(3).C_2)$ 
            \item $C_{11}^2\rtimes {\rm {SL}}_2(5)$
            \end{enumerate}
        \end{multicols}
        \end{enumerate}
    \item Let $|H|$ be odd. Then $H\cong C_3$, $G$ is a cut group and one of the following conclusions hold.
    \begin{enumerate}
        \item $K$ is a cut $2$-group having a fixed-point free automorphism of order $3$. In particular $|K|=4^a$ for some even number $a\neq 0$ and it is an extension of an abelian group of exponent dividing~$4$ by an abelian group of exponent dividing $4$.
        \item $K$ is a $7$-group having a fixed-point free automorphism of order $3$, and every cyclic subgroup of $K$ is normalized by a conjugate of $H$. Moreover, $K$ has exponent $7$ and it is an extension of an elementary abelian $7$-group by an elementary abelian $7$-group.
    \end{enumerate}
    \end{enumerate}
    Conversely, each of the Frobenius groups listed above is quadratic rational.
\end{Main}

Note that, in particular, the Frobenius group $G=C_{11}^2\rtimes {\rm {SL}}_2(5)$ is the only non-solvable quadratic rational Frobenius group. We refer the reader to the second paragraph of Section~5 for further details about the groups in (d)(ii). 

As explained in Remark~\ref{semiratcompleted}, we will also see that the class of semi-rational Frobenius groups actually coincides with that of quadratic rational Frobenius groups; therefore, the Main Theorem above also completes the analysis carried out in \cite{semiratfrobenius}. 

In the following discussion, every group is tacitly assumed to be a finite group.

\section{Some properties of quadratic rational groups}
In what follows, for a positive integer $n$, we will denote by $\Q_n$ the $n$th cyclotomic field, i.e., the subfield of $\C$ generated by $\Q$ and a primitive $n$th root of unity. As another relevant notation, given a group $G$ and an element $x\in G$, we set $\bg G x=\norm G {\langle x \rangle}/\cent G x$. Observe that conjugation in $G$ defines a faithful permutation action of $\bg G x$ on the set of generators of $\langle x\rangle$.
Also, $\Q(x)$ is contained in $\Q_{|x|}$ and it is not difficult to check that $\bg G x$ is isomorphic to $\Gal(\Q_{|x|}/\Q(x))$. 

We start by recalling some definitions already mentioned in the Introduction, that are fundamental in this work. 

\begin{definition}
    Let $G$ be a group, $x$ an element of $G$ and $r_x$ an integer. We say that $x$ is \emph{$r_x$-semi-rational} (or simply \emph{semi-rational}, when there is no need to emphasize $r_x$) in $G$ if, for every integer $j$ coprime to $|x|$, the element $x^j$ is conjugate either to $x$ or to $x^{r_x}$ in $G$. This is equivalent to the fact that $\bg G x$ induces at most two orbits, represented by $x$ and $x^{r_x}$, in its action on the set of generators of $\langle x\rangle$. 
    If every element of $G$ is semi-rational in $G$, then $G$ is called semi-rational. In this case, if the integers $r_x$ can be chosen to be all the same (say, $r_x=r$ for every $x\in G$), then $G$ is called \emph{$r$-semi-rational} (or \emph{uniformly semi-rational} if there is no need to emphasize $r$).  An important family of uniformly semi-rational groups are the $-1$-semi-rational groups, also called \emph{inverse semi-rational groups}; these are precisely the cut groups.
\end{definition}

\begin{rem}\label{rationality}    
Let $G$ be a group, and let $n$ denote the exponent of $G$. Denoting by $x^G$ the conjugacy class of $x\in G$, 
define $$\rg_G=\{j\in \mathcal{U}(\Z/n\Z)\mid x^j\in x^G\; {\textnormal{ for every }} x\in G\};$$ clearly $\rg_G$ is a subgroup of $\mathcal{U}(\Z/n\Z)$ and, setting $\Q(G)=\Q(\chi(x)\mid x\in G,\;\chi\in\irr G)$, it is not difficult to check that $\rg_G$ is isomorphic to $\Gal(\Q_{n}/\Q(G))$. We call $\rg_G$ the \emph{rationality} of $G$, and $G$ is a rational group if and only if $\rg_G$ is the full group $\mathcal{U}(\Z/n\Z)$. Now, if $r$ is an integer such that $G$ is $r$-semi-rational, then it is easy to see that the set $$\sg_G=\{s\in \mathcal{U}(\Z/n\Z)\mid {\text{ $G$ is $s$-semi-rational}}\}$$ coincides with the coset of $r$ modulo $\rg_G$ (in both the definitions of $\rg_G$ and $\sg_G$ we slightly abuse the notation identifying an integer with its congruence class modulo $n$). In this case, we refer to this coset as to the \emph{semi-rationality} of $G$.
\end{rem} 

\begin{definition}
    Let $G$ be a group. We say that $G$ is \emph{quadratic rational} if, for every irreducible (complex) character $\chi$, the field of values $\Q(\chi)=\Q(\chi(g)\mid g\in G)$ is a quadratic field extension of $\Q$, i.e. $|\Q(\chi):\Q|\leq 2$.
\end{definition}

\begin{rem}\label{RingCharacterization}
As mentioned in the Introduction, quadratic rational groups have a characterization in terms of central units of their group ring. Namely, using a variation of the argument in \cite[Theorem]{rs}, it can be seen that the group $G$ is quadratic rational if and only if the rank of the free part of $\zent{{\mathcal{U}}(\Z G)}$ equals the number of irreducible characters $\chi$ of $G$ such that $\Q(\chi)\neq \Q$ is a real extension of $\Q$. We will not make use of this characterization throughout the present work.
\end{rem}

By \cite{rs}, inverse semi-rational groups are also quadratic rational (and clearly semi-rational, by definition). On the other hand, there are quadratic rational groups that are not even semi-rational and semi-rational groups that are not quadratic rational (see the paragraph following Proposition~\ref{twogroups}). However, when the order of the group is odd, the three properties are in fact equivalent. 

\begin{proposition} \label{oddequiv}
    Let $G$ be a group of odd order. Then the following properties are equivalent.
    \begin{enumerate}
        \item $G$ is semi-rational.
        \item $G$ is quadratic rational.
        \item $G$ is inverse semi-rational.
    \end{enumerate}
\end{proposition}
\begin{proof}
    As already observed, $(3)$ implies $(1)$ and $(2)$. Also, \cite[Remark 13]{CD} proves that $(1)$ implies  $(3)$, so we need only to prove that $(2)$ implies $(3)$.

    By \cite[Exercise 3.16]{isaacscharacter} $G$ does not have any non-principal real-valued irreducible character; hence, $G$ being quadratic rational, we have $\Q(\chi)\cap \R=\Q$ for every $\chi\in \irr G$. The conclusion now follows by \cite[Proposition 2.2]{bac}.
\end{proof}

The next lemma collects two straightforward properties of quadratic rational groups.

\begin{proposition} \label{qrfacts}
    Let $G$ be a group.
    \begin{enumerate}
        \item If $G$ is quadratic rational and $N$ is a normal subgroup of $G$, then $G/N$ is quadratic rational.
        \item If $G$ is abelian, then $G$ is quadratic rational if and only if the orders of the elements of $G$ belong to the set $\{1,2,3,4,6\}$.       
    \end{enumerate}
\end{proposition}
\begin{proof}
Property (1) simply follows from the fact that every character in $\irr{G/N}$ can be viewed as an irreducible character of $G$, by inflation. As for (2), assume that $G$ is abelian. Then, for every $g\in G$, there exists an irreducible character $\lambda$ of $G$ such that  $|\Q (\lambda):\Q|=\phi (|g|)$ (where $\phi$ denotes the Euler function), hence we have $\phi(|g|)\leq 2$. As a consequence we get $|g|\in \{1,2,3,4,6\}$. Conversely, $G$ being abelian, there exists $g\in G$ such that the exponent of $G$ is $|g|$ and, since for every $\lambda\in\irr G$ we have $\Q(\lambda)\subseteq\Q_{|g|}$, the desired conclusion follows from the fact that $\phi(|g|)\leq 2$. 
\end{proof}

Next, we state a theorem by J.F. Tent (see \cite[Theorem A]{tentquadraticrat}) concerning the set of prime divisors of the order of a solvable quadratic rational group.

\begin{theorem}\label{spectraquadraticrational}
    Let $G$ be a solvable quadratic rational group. Then $\pi(G)\subseteq \{2,3,5,7,13\}$.
\end{theorem}

In Lemma~\ref{stimeqr} we derive, as an immediate consequence of Lemma~\ref{qrfields}, some information on the sizes of the normalizers of elements in quadratic rational groups.  

\begin{lemma} \label{qrfields}
    Let $K$ be a subfield of $\C$ generated by elements of degree at most $2$ over $\Q$, and let $d$ be a positive integer. Then $|K\cap \Q_d : \Q|$ divides $2^{|\pi (d)|}$ if the $2$-part of $d$ divides $4$, otherwise $|K\cap \Q_d : \Q|$ divides $2^{|\pi (d)|+1}$.
\end{lemma}
\begin{proof}
    This follows from the proof of \cite[Lemma 8]{CD}.
\end{proof}

\begin{lemma} \label{stimeqr}
Let $G$ be a quadratic rational group, and let $x\in G$. If the $2$-part of $|x|$ divides $4$ then $\phi(|x|)$ divides $ 2^{|\pi(x)|}\cdot | \bg G { x } |$, otherwise $\phi(|x|)$ divides $ 2^{|\pi(x)|+1}\cdot |\bg G { x }|$.
\end{lemma}

\begin{proof}
    Setting $d=|x|$ and $m=|\pi(d)|$, consider the field $K=\Q(x)\subseteq \Q_d$. Since $K$ is generated by quadratic elements (namely, the values $\chi(x)$ for $\chi\in\irr G$), we can apply Lemma~\ref{qrfields} to $K$ and $d$.
    If the $2$-part of $d$ divides $4$ then $|\Q(x):\Q|$ divides $2^m$, hence $\phi(d)$ divides $2^m\cdot|\Q_{d}: \Q(x)|$; but since $\Gal(\Q_{d}/\Q(x))\cong \bg G x$, the claim follows. The complementary case can be treated similarly. 
\end{proof}

It will be also useful to introduce a generalization of quadratic rationality and semi-rationality which involves normal subgroups. 

\begin{definition}
    A normal subgroup $N$ of $G$ is called \emph{quadratic rational in $G$} if, for every \(\theta\in\irr N\), we have $|\Q(\theta^G):\Q|\leq 2$.
\end{definition}

\begin{definition}
    A normal subgroup $N$ of $G$ is called \emph{semi-rational in $G$} if, for every $x \in N$,  we have $|\Q(x):\Q|\leq 2$.
\end{definition}

Observe that every normal subgroup of a semi-rational group is clearly semi-rational in $G$. In the following proposition we see that a similar situation holds for quadratic rationality. As customary, for $N\unlhd G$ and $\theta\in\irr N$, we denote by $\irr{G|\theta}$ the set of the irreducible constituents of the induced character $\theta^G$.
 
\begin{proposition}\label{normalinqr}
    Let $G$ be a group and $N$ a normal subgroup of $G$. Then the following conclusions hold.
    \begin{enumerate}
    \item For every $\theta\in\irr N$, $\chi\in\irr{G|\theta}$ and $x\in N$, we have $\Q(\theta^G(x))=\Q(\chi(x))$ and $\Q(\theta^G)\subseteq\Q(\chi)$.
    \item If $G$ is quadratic rational, then $N$ is quadratic rational in $G$. 
    \end{enumerate}
\end{proposition}
\begin{proof}
Let $\theta\in \irr N$ and $\chi\in\irr{G|\theta}$. Also, let $I$ be the inertia subgroup $I_G(\theta)$ of $\theta$ in $G$, let $T$ be a right transversal for $I$ in $G$, and define $e=[\chi,\theta^G]$. By Clifford's Theory, for $x\in N$ we have $\chi(x)=e \sum_{t\in T}\theta^{t}(x)$; on the other hand we have $\theta^G(x)=|I/N|\sum_{t\in T}\theta^{t}(x)$, so it easily follows that $\Q(\theta^G(x))=\Q(\chi(x))$. Taking into account that $\theta^G(g)=0$ for every $g\in G\setminus N$, we get $\Q(\theta^G)\subseteq\Q(\chi)$ and the first conclusion is proved. The second conclusion follows at once.
\end{proof}

Next, we recall some techniques that are relevant when working with quadratic rational groups (see~\cite{navarrotent});  these techniques use the concept of the so called semi-inertia subgroup of a character.

\begin{definition}\label{semi}
Let $N\normal G$ and $\theta\in \irr N$. The \emph{semi-inertia subgroup} of $\theta$ in $G$ is the subgroup $$I^*_\theta=\{h\in G \ | \  \theta^h=\theta^\sigma \text{ for some }\sigma \in \Gal(\Q(\theta)/\Q)\}.$$
\end{definition}

Observe that $I^*_\theta$ contains the inertia subgroup $I_\theta=I_G(\theta)$ as a normal subgroup.

\begin{lemma}\label{lemmant}
    Let $N\normal G$, $\theta\in\irr N$ and $I_{\theta},I^*_{\theta}$ as in Definition~{\rm \ref{semi}}. Then, setting $g\mapsto \sigma$ where $\theta^g=\theta^\sigma$, we get a surjective group homomorphism from $I^*_{\theta}$ to $\Gal(\Q(\theta)/\Q(\theta^G))$ with kernel $I_{\theta}$. In particular, $\theta^G$ is rational if and only if $|I_\theta^*/I_\theta|=|\Gal(\Q(\theta)/\Q)|$, and this happens if there is some character $\chi\in \irr{G|\theta}$ that is rational.
\end{lemma}
\begin{proof}
    See \cite[Lemma 2.3]{navarrotent}.
\end{proof}

The following results are useful to produce semi-rational elements in quadratic rational groups.

\begin{lemma} \label{semiratpel}
 Let $G$ be a quadratic rational group and $x\in G$. If the order of $x$ is $p^n$ or $2p^n$ where $p$ is an odd prime number, then $x$ is semi-rational in $G$.
\end{lemma}
\begin{proof}
Since $G$ is quadratic rational, the field $K=\Q(x)\subseteq \Q_{|x|}=\Q_{p^n}$ is generated by complex numbers having degree at most $2$ over $\Q$. So we can apply Lemma \ref{qrfields}, obtaining that $|\Q(x):\Q|$ divides $2^{|\pi(p^n)|}=2$.
\end{proof}

 \begin{proposition}    
 \label{normalcycl}
    Let $G$ be a quadratic rational group and let $N$ be a cyclic normal subgroup of $G$. Then every generator of $N$ is semi-rational in $G$.
\end{proposition}
\begin{proof}

Let $y$ be a generator of $N$, and set $m=|y|$ (we can clearly assume $m>2$). Since the dual group $\irr N$ is cyclic of order $m$, there exists $\theta\in\irr N$ such that $\Q(\theta)=\Q_{m}$. Observe that, for every $g\in G$, $\theta^g$ lies in the subgroup of $\irr N$ generated by $\theta$, so there exists an integer $k$ coprime to $m$ such that $\theta^g=\theta^k$. In other words, we have $I^*_\theta=G$. Moreover, by Proposition \ref{normalinqr} we have $|\Q(\theta^G):\Q|\leq 2$, hence $\frac{\phi(m)}{2}$ divides $|\Q_m:\Q(\theta^G)|$ which is the same as $|I_\theta^*/I_\theta|=|G/I_\theta|$ by Lemma \ref{lemmant}. Since $\cent G y\subseteq I_\theta$, we see that $\frac{\phi(m)}{2}$ divides $|G/\cent G y|=|\bg G y|=|\Q_m:\Q(y)|$, so $|\Q(y):\Q|\leq 2$ as wanted.
\end{proof}

\begin{rem}\label{semiratorder}
 The existence of a semi-rational element in a group $G$ yields some information on the structure of $G$, in particular, on the existence of elements in $G$ having certain orders. This follows from the observation that, if $x\in G$ is a semi-rational element, then the section $\bg G x$ of $G$ is isomorphic to a subgroup of index at most $2$ of $\aut{\langle x\rangle}$: in fact, $\bg G x\cong \Gal(\Q_{|x|}/\Q(x))\subseteq\Gal(\Q_{|x|}/\Q)\cong\aut{\langle x\rangle}$, and $|\Q(x):\Q|\leq 2$.
 
 For instance, assume that $G$ is a quadratic rational group \emph{of odd order}. Then $G$ is inverse semi-rational by Proposition~\ref{oddequiv} and, since it is solvable, Theorem~1.2 of \cite{bac} yields $\pi(G)\subseteq\{3,5,7\}$. But if $5\in\pi(G)$ then, considering a (semi-rational) element $x\in G$, in view of the paragraph above we should have a section of $G$ having order $2$, clearly a contradiction.  Also, if $y\in G$ is an element of order $21$, then $G$ should have a section isomorphic to a subgroup of order $6$ of $\aut{\langle y\rangle}\cong C_2\times C_6$, and again we get the contradiction that $G$ has an element of order $6$; for similar reasons, $G$ cannot be a $7$-group. To sum up, $G$ is a $\{3,7\}$-group (but not a $7$-group) with no elements of order $21$; note that if $\pi(G)=\{3,7\}$ then, by a well-known result of K. Gruenberg and O. Kegel, $G$ is either a Frobenius or a $2$-Frobenius group (see \cite[Theorem~2.4]{gkdelrio}), and both cases can actually occur in view of Proposition~5.2 and Proposition~5.3 of \cite{gkdelrio}. 
     
Assume now that $G$ as above has exponent $21$ (this happens for instance if $G$ is a Frobenius group, as will be shown in (2)(b) of the Main Theorem). Working with elements $x$ and $y$ of $G$ having orders $3$ and $7$ respectively, and taking into account that $\bg G x$ and $\bg G y$ have orders $1$ and $3$ respectively, it is not difficult to see that any integer $j$ such that $x^j\in x^G$ and $y^j\in y^G$ must be congruent to $1$ modulo $3$ and congruent to $1$, $2$ or $4$ modulo $7$; in other words, $j$ must be congruent to $1$, $4$ or $16$ modulo $21$. Actually, it can be easily seen that $\rg_G=\langle 4 \rangle$ in ${\mathcal{U}}(\Z/21\Z)$ and hence, $G$ being inverse semi-rational, we also get  $\sg_G=-1\cdot \langle 4 \rangle$. 
\end{rem}

\section{Complements of quadratic rational Frobenius groups}

In this section we provide a list of groups that possibly arise as Frobenius complements of quadratic rational Frobenius groups (see Theorem~\ref{qrcompl}, dealing with the solvable case, and Proposition~\ref{nonsolvablefrobcompl} for the non-solvable case). As we will see in proving the Main Theorem, every group in this list except ${\rm{SL}}_2(5).C_2$ is indeed a Frobenius complement of some quadratic rational Frobenius group. We start by gathering some general facts concerning the structure of Frobenius groups.

\begin{lemma}\label{frobeniuscomplementsylows}
    Let $H$ be a Frobenius complement and $p,q\in \pi(H)$. Then the following conclusions hold.
    \begin{enumerate}
        \item Every subgroup of $H$ of order $pq$ is cyclic.
        \item If $|H|$ is even then $H$ has a unique involution, which is therefore central.
        \item The Sylow $2$-subgroups of $H$ are either cyclic or generalized quaternion, and the Sylow $p$-subgroups of $H$ with $p\neq 2$ are cyclic.
    \end{enumerate}
\end{lemma}
\begin{proof}
    See Theorem 18.1 of \cite{passpermgroup}.
\end{proof}

\begin{lemma}\label{passman18.3}
    Let $H$ be a solvable Frobenius complement such that $\zent H$ is a non-trivial $2$-group. Then the following conclusions hold.
\begin{enumerate}
    \item Let $p>2$ and $P\in \syl p H$; then $P\subseteq H'$ unless $|P|=3$ and the Sylow $2$-subgroups of $H$ are isomorphic to $Q_8$.
    \item The Hall $2'$-subgroups of $H'$ are cyclic.
    \item Let $M$ be a Hall $2'$-subgroup of $H$. If $\oh 2 H\not \cong Q_8$ then $M\normal H$. Otherwise either $M\normal H$ or there exists $T\leq M$ such that $[M:T]=3$ and $T\normal H$.
\end{enumerate}
\end{lemma}
\begin{proof}
    See Lemma $18.3$ \cite{passpermgroup}.
\end{proof}

Next, we focus more specifically on the context of quadratic rational groups. In particular, in Proposition~\ref{sylowqrfrcmpl} (that is refined in Proposition~\ref{notQ32}) the possible structure of Sylow subgroups of a quadratic rational Frobenius complement is determined. This should be compared with Lemma~{2.5} of \cite{semiratfrobenius}.

\begin{proposition}\label{qrfrobthm}
    Let $G$ be a Frobenius group with complement $H$ and kernel $K$. Then $G$ is quadratic rational if and only if the following properties hold:
    \begin{enumerate}
        \item H is quadratic rational.
        \item K is quadratic rational in G.
    \end{enumerate}
\end{proposition}
\begin{proof}
    The ``only if" part of the statement easily follows from the observation that $H$ is a factor group of $G$, and by Proposition~\ref{normalinqr} (or simply by the fact that every irreducible character of $K$ induces irreducibly to $G$). As for the other implication, this follows from \cite[Theorem 18.7]{huppertcharacters}.
\end{proof}

\begin{proposition}\label{sylowqrfrcmpl}
    Let $H$ be a quadratic rational Frobenius complement. Then the following conclusions hold.
    \begin{enumerate}
        \item For $|H|$ even, a Sylow $2$-subgroup of $H$ is isomorphic to one among the groups $C_2$, $C_4$, $Q_8$,   $Q_{16}$, $Q_{32}$.
        \item A Sylow $p$-subgroup of $H$ for any odd prime divisor $p$ of $|H|$ is isomorphic to $C_p$.
    \end{enumerate}
\end{proposition}
    \begin{proof} We start by proving claim (1). Assume first that $P\in \text{Syl}_2(H)$ is abelian, hence cyclic by Lemma~\ref{frobeniuscomplementsylows}. Then $H$ has a normal $2$-complement $N$, and $P\cong H/N$ is therefore quadratic rational. By Proposition~\ref{qrfacts} we have $P\cong C_2 $ or $P\cong C_4$.
            
Suppose now $P\cong Q_{2^{n+1}}$ and assume $n\geq 3$. Take $x\in P$ of order $2^{n}$. Then we have $2^{n+1}\geq |\norm H {\langle x \rangle}|_2\geq 2^{2n-3}$ by Lemma~\ref{stimeqr}, therefore $n\leq 4$.
            
As regards claim (2), by Lemma~\ref{frobeniuscomplementsylows} we have that $P\in \syl p H$ is cyclic. Let $x$ be a generator of $P$ with $|x|=p^m$: Lemma~\ref{semiratpel} yields that $x$ is semi-rational in $H$ and, since $P$ is contained in $\cent G x$, we have that $\bg G x$ is a $p'$-group. But now  $p^{m-1}$ divides $\frac{\phi(|x|)}{2}$ which in turn divides $|\bg G x|$, implying the desired conclusion $m=1$. 
\end{proof}

\begin{lemma}
\label{2centerqr}
Let $H$ be a solvable Frobenius complement such that $\zent H$ is a non-trivial $2$-group. If $p$ is an odd prime and $P\in \syl p H$, then either $P\subseteq H'$ or $p=3$ and $H\cong {\rm {SL}}_2(3)$.

\end{lemma}
\begin{proof}
    Observe that, by Lemma \ref{passman18.3}, we have $P\subseteq H'$ unless $|P|=3$ and a Sylow $2$-subgroup $Q$ of $H$ is isomorphic to $Q_8$. 
    In this situation we can apply \cite[Lemma 12.10]{passpermgroup}, obtaining that $H$ has a normal $3$-complement $N$. Now we apply \cite[Lemma 18.4]{passpermgroup} and it is not difficult to check that, among the possible structures of $H$ given by that lemma, we are left with two cases: either $H\cong {\rm SL_2(3)}$ as we want, or $H$ has a normal $2$-complement $M$ and $Q/Q'$ acts non-trivially on each Sylow subgroup of $M$.
    We only have to show the latter possibility yields a contradiction. 
    In fact, setting $T=M\cap N\normal H$ we have $H/T\cong C_3\times Q_8$ and $H\cong T\rtimes (C_3\times Q_8)$, but this contradicts the fact that $Q/Q'$ acts non-trivially on~$P$.
\end{proof}

\begin{proposition}\label{23qrcenterfrcpl}
    Let $H$ be a solvable quadratic rational Frobenius complement of even order, and assume that $H$ is non-abelian. If $\zent H$ is not a $2$-group, then $H\cong C_3\times Q_8$.
\end{proposition}
\begin{proof}
   Let $p\neq 2$ be a prime divisor of $|\zent H|$: then, by Proposition~\ref{sylowqrfrcmpl}, a Sylow $p$-subgroup $P$ of $\zent H$ is in fact a Sylow $p$-subgroup of $H$. Thus $P$ is a direct factor of $H$, hence a factor group of $H$, and it is therefore quadratic rational. This can happen only when $p=3$ by Proposition~\ref{qrfacts}. Since, by Lemma~\ref{frobeniuscomplementsylows}, $H$ has a central involution, we have $\pi(\zent H)=\{2,3\}$. Therefore $H=P\times K$, where $K$ is a solvable quadratic rational Frobenius complement with $3\nmid|K| $, and $\zent K$ is  a $2$-group. 
    
    By Lemma \ref{passman18.3} we have $K=M \rtimes S$, where $S\in \text{Syl}_2(H)$ and $M=\langle x\rangle$ is a cyclic $2$-complement. Since $K$ is quadratic rational, $x$ is semi-rational in $K$ by Proposition \ref{normalcycl}. Therefore, $\phi(|x|)$ divides $2|\bg G x|$ which in turn divides $2|K|/|x|=2|S|$: by Theorem~\ref{spectraquadraticrational} and Proposition~\ref{sylowqrfrcmpl}, this can happen only if $|x|$ divides $5$. 
    Since $K$ is non-abelian and $Q_{32}$ is not quadratic rational, in view of Proposition~\ref{sylowqrfrcmpl} we deduce that $K$ is isomorphic to one of the following groups: $Q_8$, $Q_{16}$, $C_5\rtimes C_2$, $C_5\rtimes C_4$, $C_5\rtimes Q_{8}$, $C_5\rtimes Q_{16}$.
  Observe that the field of values of every irreducible character of $K$ must be a subfield of $\Q_3$, since the characters of $H$ are of the form $\lambda\times \chi$ for $\lambda\in \irr{C_3} $ and $\chi\in \irr K$. But the exponent of $K$ is coprime with $3$, therefore $K$ is a rational group. Our claim now follows from the fact that, as it is not difficult to check, the only rational group in the above list is $Q_8$; for instance, if $K$ is either $C_5\rtimes Q_8$ or $C_5\rtimes Q_{16}$, an element $x$ of order $5$ cannot be rational since $\bg K x=K/\cent K x$ should be a cyclic group of order $4$, which is not the case. 
\end{proof}

\begin{proposition}\label{notQ32}
    Let $H$ be a quadratic rational Frobenius complement. If $P\in\syl 2 H$, then $P$ is not isomorphic to $Q_{32}$. 
\end{proposition}
\begin{proof}
 For a proof by contradiction, assume $P\cong Q_{32}$ and note that this implies $\oh 2 H\not\cong Q_8$,  because $Q_{32}$ does not have any normal subgroup isomorphic to $Q_8$. If $H$ is solvable, then Proposition~\ref{23qrcenterfrcpl} yields that $\zent H$ is a $2$-group, thus we can apply Lemma \ref{passman18.3} and derive that $H$ has a normal $2$-complement $M$. But $P\cong H/M$ is then quadratic rational, whereas $Q_{32}$ is not.
            On the other hand, if $H$ is non-solvable,  then by  \cite[Theorem 18.6]{passpermgroup} there exists a normal subgroup $H_0\normal H$ with $[H:H_0]\leq  2$ and such that $H_0\cong {\rm {SL}}_2(5)\times M$ where $M$ is a group of order coprime to 2, 3, and 5. As a consequence, the $2$ part of $|H|$ is at most $2^4$, yielding the final contradiction.
\end{proof}

We are now in a position to prove part of Theorem~A, stated in the Introduction. Recall that we defined 

$$H_1=\langle x,y| x^5=y^4=1, x^y=x^{-1} \rangle$$  and
$$H_2=\langle x,y,z| x^3=y^8=[x,z]=1, z^2=y^4,  y^z=y^{-1}, x^y=x^{-1}\rangle$$ (note that $H_1\cong \texttt{SmallGroup(20,1)}$ and $H_2\cong \texttt{SmallGroup(48,18)}$).

We will treat the solvable and the non-solvable case separately, in Theorem~\ref{qrcompl} and Proposition~\ref{nonsolvablefrobcompl} respectively.

\begin{theorem}    
\label{qrcompl} Let $G$ be a solvable quadratic rational Frobenius group with Frobenius complement $H$. Then $H$ is isomorphic to one of the groups in the following list: 
        $$C_2,\; C_3,\;C_4,\;C_6,\; Q_8,\; C_3\rtimes C_4,\; Q_{16},\; H_1, \;{\rm{SL}}_2(3),\; C_3\rtimes Q_8,\; C_3\times Q_8\;, H_2,\; {\rm{SL}}_2(3).C_2\,.$$
\end{theorem}
\begin{proof}
Let $H$ be as in our hypothesis. If $H$ is a $2$-group, then Proposition~\ref{sylowqrfrcmpl} and Proposition~\ref{notQ32} yield that $H$ is isomorphic to one among $C_2$, $C_4$, $Q_8$, $Q_{16}$, and we are done. On the other hand, if $H$ has odd order, then by the discussion in Remark~\ref{semiratorder} we know that $H$ is a $\{3,7\}$-group (but not a $7$-group) with no elements of order $21$, which implies (together with Lemma~\ref{frobeniuscomplementsylows} and Proposition~\ref{sylowqrfrcmpl}) that $H\cong C_3$. Also, if $H$ is abelian and not one of the groups already considered, Proposition~\ref{qrfacts} easily implies $H\cong C_6$. 

For the rest of this proof, we will then assume that $H$ is a non-abelian group of even order, and not a $2$-group. In this situation, if $\zent H$ is not a $2$-group, then Proposition~\ref{23qrcenterfrcpl} yields $H\cong  C_3\times Q_8$ and again we are done. So we can assume that $\zent H$ is a $2$-group; now, either $H\cong {\rm {SL}}_2(3)$ (and we are done), or by Lemma~\ref{2centerqr} and Lemma~\ref{passman18.3} we get that $H$ has cyclic Hall $2'$-subgroups contained in $H'$, so we will assume the latter property as well.

Let us consider first the case when $H$ has a (cyclic) normal $2$-complement $M$: again by Lemma~\ref{passman18.3}, this is what happens unless $\oh 2{H}\cong Q_8$. So we have $H=M\rtimes S$ where $S\in \Syl_2(H)$ and, by Theorem~\ref{spectraquadraticrational} and Proposition~\ref{sylowqrfrcmpl}, the order of $M$ is a divisor of $3\cdot 5\cdot 7\cdot 13$. 

If $M$ has an element $x$ of order $7$ or $13$ then, by Proposition~\ref{normalcycl} and Remark~\ref{semiratorder}, $\bg H x$ should have elements of order $3$, contradicting the fact that in our situation $\bg H x$ is clearly a $2$-group.
    
 Suppose now $|M|=15$ and observe that, $M$ being cyclic, the factor group $S/\cent S M$ is abelian; in particular, $S'\subseteq\cent S M$ is normal in $H$, and the factor group $H/S'\cong M\rtimes S/S'$ is quadratic rational. 
    Recalling that $S$ is isomorphic to one of $C_2$, $C_4$, $Q_8$ or $Q_{16}$, if $S/S'$ is cyclic then $S$ is cyclic, so the only possibilities are $H\cong C_{15}\rtimes C_2$ and $H\cong C_{15}\rtimes C_4$. Since the elements of order $15$ are semi-rational in $H$ by Proposition~\ref{normalcycl}, we have that $4= \phi(15)/2$ divides $|H/\cent H M|$, which forces $S\cong C_4$ and $\cent S M=1$.  But in this case the group $H$ does not have a central involution, against Lemma~\ref{frobeniuscomplementsylows}.
    Consider then the case $S/S'\cong C_2\times C_2$, which implies that $S'$ is isomorphic to either $C_2$ or $C_4$. If $S'\cong C_4$, then $H$ has a cyclic normal subgroup generated by an element $y$ of order $60$, and this $y$ is semi-rational in $H$ by Proposition~\ref{normalcycl}, so $8$ divides $|\bg H y|$; on the other hand, $\bg H y$ is isomorphic to a factor group of $ S/S'\cong C_2\times C_2$, a clear contradiction.
    We are left with the case $H=M\rtimes Q_8$. In this situation, denoting by $x$ a generator of $M$ and taking into account that again $\bg H x\cong S/S'\cong C_2\times C_2$, it is easy to see that the structure of $H/S'$ is the following: $$H/S'\cong M\rtimes (C_2\times C_2)=\langle x \rangle \rtimes (\langle a \rangle \times \langle b\rangle),$$ where $x^a=x^{11}$,  $x^b=x^4$. But then $H$ is isomorphic to $\texttt{SmallGroup(120,14)}$, which is not quadratic rational. Therefore $M$ is isomorphic to either $C_3$ or $C_5$. 

Assume $M\cong C_5$. Then $S\not\cong C_2$, because ${\rm D}_{10}$ does not have a central involution and $C_{10}$ has a center that is not a $2$-group. If $S\cong C_4$ then, arguing similarly, we are left with the only possibility $H\cong H_1$ which, as previously observed, is isomorphic to the group $\texttt{SmallGroup(20,1)}$. Now let us assume $S=Q_8$: it can be checked that $\texttt{SmallGroup(40,4)}\cong C_5\rtimes Q_8$ is the only group of the relevant kind in which the elements of order $5$ are semi-rational, but this group is not quadratic rational.
    Finally, we can check that every group of the form $C_5\rtimes Q_{16}$ is not quadratic rational (in particular, not every generator of the normal cyclic subgroup of order $20$ is semi-rational). 
       
Now let us consider the case $M\cong C_3$. Every possible structure is easily discarded except $C_3\rtimes C_4$ (with the inversion action), \texttt{SmallGroup(24,4)} and $H_2=\,$\texttt{SmallGroup(48,18)} (whose structures are $C_3\rtimes Q_8$ and $C_3\rtimes Q_{16}$ respectively). All of them are quadratic rational groups.
    
\smallskip
Finally, still assuming that $H$ is not a $2$-group, $\zent H$ is a $2$-group and the Hall $2'$-subgroups of $H$ are cyclic, it remains to consider the case when $H$ does not have a normal $2$-complement, which implies $\oh 2 H\cong Q_8$. Denoting by $T$ a Hall $\{2,3\}'$-subgroup of $H$, by Lemma~\ref{passman18.3} and Theorem~\ref{spectraquadraticrational}, we have $$H\cong T\rtimes U$$ where $T$ is cyclic of order a divisor of $5\cdot 7\cdot 13$, and $U$ is a Hall $\{2,3\}$-subgroup of $H$. But in fact, $T$ cannot have any element $x$ of order $7$ or $13$, as otherwise $x$ would be semi-rational in $H$ and $\bg H x=H/\cent H x$ should have an order divisible by $3$, against the fact that $\cent H x$ contains a Hall $2'$-subgroup of~$H$. So $T$ can be either trivial or a group of order $5$.

If $T=1$ and $\oh 2 H\cong Q_8$ is a Sylow $2$-subgroup of $H$, then we get a contradiction because all the possible structures of $H$ can be easily discarded: in ${\rm{SL}}_2(3)$ the Hall $2'$-subgroups are not contained in the derived subgroup, whereas $C_3\times Q_8$ has a normal $2$-complement. Assuming now that $T=1$ and the Sylow $2$-subgroups of $H$ are isomorphic to $Q_{16}$, again the fact that $H$ does not have a normal $2$-complement yields that the only possible structure of $H$ is ${\rm{SL}}_2(3).C_2$ (which is indeed a quadratic rational group), as wanted.

To conclude, assume $|T|=5$. Then $3$ must divide $|U|$ (otherwise $H$ has a normal $2$-complement) and, since $\cent H T$ contains both a full $2$-complement of $H$ and $\oh 2 H$, the Sylow $2$-subgroups of $H$ must be isomorphic to $Q_{16}$ (otherwise the center of $H$ would not be a $2$-group). To sum up, the only possible structure of $H$ is $C_5\rtimes({\rm{SL}}_2(3).C_2)$, i.e., $H$ must be isomorphic to \texttt{SmallGroup(240,105)}. But this group is not quadratic rational, and the proof is complete.
\end{proof}

\begin{proposition}\label{nonsolvablefrobcompl}
Let $H$ be a non-solvable quadratic rational Frobenius complement. Then $H$ is isomorphic either to ${\rm {SL}}_2(5)$ or to ${\rm {SL}}_2(5).C_2\,.$
\end{proposition}
\begin{proof}
    If $H$ is non-solvable then, by \cite[Theorem 18.6]{passpermgroup}, $H$ has a normal subgroup $H_0\cong {\rm {SL}}_2(5)\times M$ where $M$ is a group whose order is coprime with $2,3,5$, and $[H:H_0]\leq 2$. If $M\neq 1$ then its order $m$ is a product of pairwise distinct (odd) primes by Proposition~\ref{sylowqrfrcmpl}; since, by Lemma~\ref{frobeniuscomplementsylows}, every subgroup of $H$ whose order is the product of two distinct primes is cyclic, it is easy to see that $M$ is cyclic and we can choose $\chi\in\irr M$ such that $\Q(\chi)=\Q_m$. Now, consider $\theta=1_{{\rm {SL}}_2(5)}\times \chi\in\irr{H_0}$: since $H$ is quadratic rational and $[H:H_0]\leq 2$, taking into account Lemma~\ref{lemmant} we have $|\Q(\theta):\Q|\leq 4$. But this forces $m$ to be coprime with every prime larger than $5$, whence $M=1$ and the conclusion follows.
\end{proof}

\begin{rem}\label{iranian}
    In Theorem~\ref{qrcompl} and Proposition~\ref{nonsolvablefrobcompl} we obtained a list of all the quadratic rational Frobenius complements, and it can be checked (for instance with \texttt{GAP}) that every group $H$ in this list is $r$-semi-rational for a suitable $r$ such that $r^4$ is congruent to $1$ modulo the exponent of $H$. As we will see in the Main Theorem, this property holds indeed for every quadratic rational Frobenius group.     
Conversely, the possible semi-rational Frobenius complements are listed in \cite{semiratfrobenius}, and it turns out that all of them appear in Theorem~\ref{qrcompl} and Proposition~\ref{nonsolvablefrobcompl} (we take the occasion to point out that the group $C_{15}\rtimes C_4$ is not a semi-rational Frobenius complement, so it should be removed from \cite{semiratfrobenius}. In fact, if an element $x$ of order $15$ is semi-rational in a group of this kind, then the centralizer of $x$ would be $\langle x\rangle$ and the group would not have a central involution). 
\end{rem}

\section{On quadratic rational Frobenius groups with abelian kernel}

In this section we will discuss the possible structure of a quadratic rational Frobenius group \emph{with abelian kernel}. 
Since the kernel of a Frobenius group having a complement of even order is abelian (see \cite[Theorem V.8.18]{huppertgroupsI}), in view of Theorem~\ref{qrcompl} and Proposition~\ref{nonsolvablefrobcompl} this analysis will cover essentially all the relevant cases (namely, except when the Frobenius complement is isomorphic to $C_3$ and the Frobenius kernel is non-abelian). In particular, we will see in Proposition~\ref{prop5.5} that, for a Frobenius group whose kernel is an elementary abelian $p$-group, the properties of semi-rationality, quadratic rationality and uniform semi-rationality are equivalent.

The following proposition yields a significant restriction on the exponent of the Frobenius kernel, when dealing with a quadratic rational Frobenius group whose kernel is abelian. We include the case $p=2$ for completeness, although it can be also deduced from Proposition~\ref{twogroups}.

\begin{proposition}\label{prop4.5}
    Let $G$ be a quadratic rational Frobenius group with abelian kernel $K$. Let $P\in\syl p K$ for some prime $p$ that divides $|K|$. Then the following conclusions hold.
    \begin{enumerate}
        \item If $p>2$ then $P$ is an elementary abelian $p$-group.
        \item If $p=2$ then the exponent of $P$ divides $4$.
    \end{enumerate}
\end{proposition}
\begin{proof}
    Let $H$ be a Frobenius complement of $G$, and  let $x$ be a non-trivial element of $P$ having order $p^n$ for a suitable positive integer $n$. Then $\cent G x=K$, and therefore $\bg G x$ is isomorphic to a subgroup of $H$; in particular, since $|H|$ and $|K|$ are coprime, $p$ does not divide $|\bg G x|$. If $p>2$ then, by Lemma~\ref{stimeqr}, we have $p^{n-1}\cdot (p-1)\mid 2\cdot|\bg G x|$. But now the coprimality of $|\bg G x|$ with $p$ forces $n= 1$ and we are done. 

    Assume now $p=2$; in this case, $H$ must have odd order and so we get $H\cong C_3$ by Theorem~\ref{qrcompl}. Let $2^n$ be the exponent of $P$, and take $\lambda\in\irr P$ such that $\Q(\lambda)=\Q_{2^n}$. Observe that, since $K$ is contained in the inertia subgroup $I_\lambda$, then $I^*_\lambda/I_\lambda$ is isomorphic to a subgroup of $C_3$. But $I^*_\lambda/I_\lambda$ is isomorphic to the $2$-group $\Gal(\Q(\lambda)/\Q(\lambda^G))$ by Lemma~\ref{lemmant}. It follows that $I^*_\lambda=I_\lambda$, hence $\Q(\lambda)=\Q(\lambda^G)$. Proposition~\ref{normalinqr} yields now that $\Q(\lambda)\subseteq \Q(\chi)$ for any $\chi\in \Irr(G|\lambda)$; this forces $|\Q(\lambda):\Q|\leq 2$, and therefore $n\leq 2$.  
\end{proof}

So, assuming the setting of Proposition~\ref{prop4.5}, if the order of the Frobenius complement $H$ is even then the kernel $K$ is the direct product of abelian groups of odd prime exponent. In the following proposition we have a closer look at this situation, providing the possible orders of the Frobenius kernel.

\begin{proposition}\label{prop4.6}
 Let $G$ be a quadratic rational Frobenius group with complement $H$ of even order and kernel $K$. Then the possible orders of $K$ are displayed in the third column of Table~$\ref{possibleK}$, according to the structure of $H$ (where $a$ and $b$ are positive integers).

    \begin{table}[h]
    
    \begin{tabular}{|c|c|c|}
    \hline
    $|H|$ & $H$ & $|K|$ \\
    \hline
    $2$& $C_2$ & $3^a,\; 5^a$ \\ 
    $4$& $C_4$ & $3^a,\; 5^b,\; 3^a\cdot 5^b\quad\textnormal{ with }a \textnormal{ even}$ \\ 
    $6$ & $C_6$ & $5^a,\; 7^b,\; 13^b\quad\textnormal{ with }a \textnormal{ even}$\\
    $8$ & $Q_8$ & $3^a,\; 5^a,\; 3^a\cdot 5^b\quad\textnormal{ with }a, b \textnormal{ even}$\\
    $12$ & $C_3\rtimes C_4$ & $5^a,\; 7^a,\; 13^b\quad\textnormal{ with }a \textnormal{ even}$\\
    $16$ & $Q_{16}$ & $3^a,\; 5^a,\; 3^a\cdot 5^b\quad\textnormal{ with }4\mid a,b$\\
    $20$ &$H_1$ & $3^a\quad\textnormal{ with }4\mid a$\\
    $24$ &${\rm {SL}}_2(3)$ & $5^a,\; 7^a,\; 13^a\quad\textnormal{ with }a \text{ even}$\\
    $24$ &$ C_3\rtimes Q_8$ & $5^a,\; 7^a,\; 13^a,\; 5^a\cdot 7^b\quad\textnormal{ with }a,b\text{ even}$\\
    $24$ &$C_3\times Q_8$ & $5^a,\; 7^a,\; 13^a,\; 5^a\cdot 7^b\quad\textnormal{ with }a,b \text{ even}$\\
    $48$ &$H_2$ & $5^a,\; 7^b,\; 13^a,\; 5^a\cdot 7^b\quad\textnormal{ with }4\mid a \textnormal{ and }b \textnormal{ even}$\\
    $48$ &${\rm {SL}}_2(3).C_2$ & $5^a,\; 7^b,\; 13^a\quad\textnormal{ with }4\mid a \textnormal{ and }b \textnormal{ even}$\\
    $120$ & ${\rm {SL}}_2(5)$ & $7^a,\;11^b,\; 13^a\quad\textnormal{ with }4\mid a\textnormal{ and } b \textnormal{ even}$\\
    $240$ & ${\rm {SL}}_2(5).C_2$ & $7^a,\; 11^a,\;13^a,\; 17^a\quad\textnormal{ with }4\mid a$\\
    \hline
    \end{tabular}
    \vspace{0.3cm}
    
    \caption{Possible orders of $K$.}
        \label{possibleK}
    \end{table}
    
\end{proposition}
\begin{proof}
        The proof mainly relies on computation, taking into account Theorem~\ref{spectraquadraticrational} (for the solvable case) and the following facts.        
Let $\lambda \in \irr K$ such that $\Q(\lambda)=\Q_n$, where $n$ denotes the exponent of $K$. Since, by Proposition~\ref{normalinqr}, the degree of $\Q(\lambda^G)$ over $\Q$ is at most $2$, we see that $H$ has a subgroup isomorphic to the group $I^*_\lambda/K=I^*_\lambda/I_\lambda\cong\Gal(\Q(\lambda)/\Q(\lambda^G))$ (recall Lemma~\ref{lemmant}), which has index at most $2$ in $\mathcal{U}(\Z/n\Z)$. Moreover, recall that $|K|\equiv 1$ mod $|H|$.

 Given that, we sketch a proof of our claim for some of the possible structures of $H$. Consider first the case $H=Q_8$, and assume for the moment that $K$ is a $p$-group for a suitable prime $p$. Then by Theorem~\ref{spectraquadraticrational} we have $p\in \{3,5,7,13\}$, but $p$ cannot be $7$ nor $13$ since there are no $3$-elements in $Q_8$. Now, for $|K|=3^a$, we have $3^a\equiv 1$ mod $|Q_8|$ if and only if $a$ is even. The same argument applies in the case $p=5$: if $|K|=5^a$, then $a$ must be even as well. On the other hand, if $|K|$ is not a prime power, then necessarily we have $|K|=3^a\cdot 5^b$ with $a$ and $b$ even; note that, since in this case the exponent of $K$ is $15$ and $\mathcal{U}(\Z/15\Z)\cong C_2\times C_4$, $H$ must have a subgroup isomorphic either to $C_2\times C_2$ or to $C_4$. The former case yields of course a contradiction, but the latter can actually occur, so $|K|=3^a\cdot 5^b$ is not ruled out at this stage.
    
 Assume next $H={\rm SL}_2(3)$. In this case, using the argument above, we cannot exclude any of the primes that possibly divide the order of a solvable quadratic rational group (and that are coprime to $|{\rm SL}_2(3)|$), because ${\rm SL}_2(3)$ has elements of order $6$. But we can show that $K$ is necessarily of prime-power order; in fact, ${\rm SL}_2(3)$ does not have any abelian subgroup of order $\phi(pq)/2$ for any choice of $p\neq q$ in $\{5,7,13\}$. As above,  the exponent $a$ must be even.
    
 Finally, let us consider $H={\rm SL}_2(5).C_2$. Then the set of orders of the elements in $H$ is $S=\{ 1,2, 3, 4, 5, 6, 8, 10, 12\}$, so the primes $p$ that satisfy $\frac{p-1}{2}\in S$ and $p\nmid|H|$ are those in the set $\{7,11,13,17\}$. The fact that $K$ is necessarily a $p$-group can be seen as above, as well as the condition that $4$ divides the exponent $a$.
\end{proof}

Note that the Sylow subgroups of the Frobenius kernel $K$ are all normalized by $H$, and the action of $H$ on each of them is a Frobenius action.
The next step is to consider these actions, taking into account that, by Proposition~\ref{prop4.5}, every Sylow $p$-subgroup of $K$ for an odd prime $p$ can be viewed as an $H$-module over the field $\F_p$ with $p$ elements.

The following concept, introduced in \cite[Definition~2.8]{farias}, will be relevant in our analysis.   

\begin{definition}
    Let $H$ be a group, and $V$ a finite-dimensional $H$-module over a finite field $\F$. Let $k$ be a positive integer that divides $|\F^\times|$. We say that $V$ has the \emph{$k$-eigenvalue property} if there exists an element $a\in \F^\times$ of order $k$ satisfying the following condition: for every $v\in V$, there exists $h\in H$ such that $h\cdot v=a v$. 
\end{definition}

\begin{theorem}\label{thm4.4}
Let $H$ be a group, $p$ a prime, and assume that $H$ acts on an elementary abelian $p$-group $K$. If $K$ (viewed as an $H$-module over $\F_p$) has the $k$-eigenvalue property, then its dual module $\irr K$ has the $k$-eigenvalue property as well. 
\end{theorem}
\begin{proof}
We will prove that, working in the semidirect product $G=K\rtimes H$, for every non-principal irreducible character $\lambda$ of $K$ there exists $h\in H\cap I_\lambda^*$ such that the coset $hI_\lambda$ has order $k$ in $I_\lambda^*/I_\lambda$. As it is not difficult to see, this is equivalent to the desired conclusion.

So, let $\lambda$ be a non-principal character in $\irr K$. We claim first that $k$ divides $|\Q(\lambda):\Q(\lambda^G)|$. In fact, let $x$ be a non-trivial element of $K$. We know that $\Q_p$ contains $\Q(x)$ and, since $K$ has the $k$-eigenvalue property as an $H$-module, there exists an element $t\in H$ such that $x^t=x^a$ where $a$ has order $k$ modulo $p$. In particular, $t\in \norm G {\langle x \rangle}$ is such that the order of $t\cent G x$ in $\bg G x$ is $k$, implying that $k\mid |\bg G x|$. As a consequence, setting $d=|\Q(x):\Q|$ and recalling that $\bg G x$ is isomorphic to ${\rm{Gal}}(\Q_{|x|}/\Q(x))$, we get $d\mid \frac{p-1}{k}$.
Now, let $\mu\in\irr K$ and $\chi\in\irr{G|\mu}$; by Proposition~\ref{normalinqr} we have $\Q(\chi(x))=\Q(\mu^G(x))$,
therefore $\Q(x)=\Q(\mu^G(x)\mid\mu\in \irr K)$. We also know that $\Q(\lambda^G)=\Q(\lambda^G(g)\mid g\in G)=\Q(\lambda^G(x)\mid x\in K)$ is generated by the values $\lambda^G(x)$ and, by the above discussion, each of these values generates a field whose degree over $\Q$ divides $d$. 
Since the group $\Gal(\Q_p/\Q)$ is cyclic, we conclude that $|\Q(\lambda^G):\Q|$ divides $d$, so $|\Q_p:\Q(\lambda^G)|=|\Q(\lambda):\Q(\lambda^G)|$ is divisible by $\frac{p-1}{d}$. 
But $d\mid \frac{p-1}{k}$ and then $k\mid \frac{p-1}{d}$.

Finally, since by Lemma~\ref{lemmant} we have that $I_\lambda^*/I_\lambda\cong\Gal(\Q(\lambda)/\Q(\lambda^G))$ is a cyclic group, and by the paragraph above the order of this group is divisible by $k$, the desired conclusion easily follows.
\end{proof}

We are ready to show that, in the class of the Frobenius groups whose kernel is an elementary abelian $p$-group, the properties of semi-rationality, quadratic rationality and uniform semi-rationality are equivalent.  It will be also useful to take into account the following general (easy) observation.

\begin{lemma}\label{easy}
    Let $G$ be a group, $r$ an integer that is coprime with the exponent of $G$, and $x\in G$. Then,  identifying $\bg G x$ with a subgroup of $\aut{\langle x \rangle}$, the following properties are equivalent.
    \begin{enumerate}
        \item $x$ is semi-rational.
        \item $[\aut{\langle x \rangle}: \bg G x]\leq 2$.
     \end{enumerate}    
        In this situation, denoting by $\tau_r$ the automorphism of $\langle x\rangle$ such that $\tau_r(x)=x^r$, $x$ is $r$-semi-rational if and only if  $\aut{\langle x \rangle}=\bg G x\langle \tau_r \rangle$.
\end{lemma}

\begin{proof}
We can clearly assume that $x$ is not a rational element. Since $\bg G x\cong \Gal(\Q_{|x|}/\Q(x))\subseteq\Gal(\Q_{|x|}/\Q)\cong\aut{\langle x \rangle}$, we get $|\Q(x):\Q|=2$ if and only if $[\aut{\langle x \rangle}: \bg G x]=2$. In other words, $x$ is semi-rational if and only if $[\aut{\langle x \rangle}: \bg G x]=2$. Moreover, if in this situation $x$ is $r$-semi-rational, then the automorphism $\tau_r$ does not lie in $\bg G x$ because $x^r$ is not conjugate to $x$ in $G$, thus we have $\aut{\langle x \rangle}=\bg G x\langle \tau_r \rangle$. Conversely, if the latter equality holds, then $x^r$ is not conjugate to $x$ in $G$ and so $x$ is $r$-semi-rational. 
\end{proof}

\begin{proposition}\label{prop5.5}
Let $G$ be a Frobenius group with complement $H$ and kernel $K$. If $K$ is an elementary abelian $p$-group for a suitable prime $p$, then the following properties are equivalent.
    \begin{enumerate}
        \item $G$ is semi-rational.
        \item $G$ is quadratic rational.
        \item $G$ is uniformly semi-rational.
    \end{enumerate}
\end{proposition}

\begin{proof}
Let us assume that (1) holds, so $G$ is semi-rational. To prove (2), by Proposition~\ref{qrfrobthm}   it will be enough to show that $H$ is quadratic rational, and that $K$ is quadratic rational in $G$. The former property holds by Remark~\ref{iranian}. As for the latter, we have to show that $|\Q(\lambda^G):\Q|\leq 2$ holds for every $\lambda\in\irr K$; in other words, assuming (as we may) $p\neq 2$, we have to show that $|I_\lambda^*/I_\lambda|=|\Q(\lambda):\Q(\lambda^G)|$ is divisible by $\frac{p-1}{2}$ for every non-principal $\lambda\in\irr K$, which means that the dual $H$-module $\irr K$ of $K$ has the $\frac{p-1}{2}$-eigenvalue property. On the other hand, by Lemma~\ref{easy}, every non-trivial element $x$ of $K$ is such that $\frac{p-1}{2}$ divides $|\bg G x|$, which means that the $H$-module $K$ has the $\frac{p-1}{2}$-eigenvalue property; the desired conclusion thus follows by Theorem~\ref{thm4.4}.

Since (3) obviously implies (1), it remains to prove that (2) implies (3). So, let $G$ be quadratic rational: again Remark~\ref{iranian} yields that the complement $H$ is $t$-semi-rational for a suitable integer $t$, and we only have to deal with the elements of order $p$ of $G$. If $x$ is such an element, then $x$ is semi-rational in $G$ by Lemma~\ref{semiratpel}. Now, if $u$ is any fixed integer which has order $p-1$ modulo $p$, an application of Lemma~\ref{easy} yields that $x$ is $u$-semi-rational. Finally, taking an integer $r$ which is congruent to $t$ modulo the exponent of $H$, and to $u$ modulo $p$, we have that $G$ is $r$-semi-rational.
\end{proof}

\begin{rem}\label{USR} Let $G$ be a quadratic rational Frobenius group with a Frobenius complement $H$ of even order, and whose kernel $K$ is an elementary abelian $p$-group for a suitable prime $p$. Then, by Proposition~\ref{prop4.6}, we have $p\in\{3,5,7,11,13,17\}$; we claim that, for each of the possible values of $p$ except $17$, the elements of $K$ are all $u$-semi-rational for a fixed integer $u$ with  $u^4\equiv 1$ modulo $p$. Given that, taking into account Remark~\ref{iranian} and using the Chinese Remainder Theorem, it is easily seen that the whole group $G$ is $r$-semi-rational for a suitable integer $r$ with $r^4\equiv 1$ modulo the exponent of~$G$.

In fact, the above claim is certainly true for $p=3$ or $p=5$, as ${\mathcal{U}}(\Z/p\Z)$ has order $2$ or $4$ respectively. Assume then $p=7$ and, using the notation of Lemma~\ref{easy}, consider the automorphism $\tau_3$ of $\langle x\rangle$ for $x\in K\setminus\{1\}$; we see that $\tau_3$ is a generator of $\aut{\langle x\rangle}$, so $x$ is $3$-semi-rational by Lemma~\ref{easy}. Moreover, the decomposition of $\tau_3$ in its $2$- and $2'$-part is given by $\tau_3=\tau_6\cdot\tau_4$ and, since $\bg G x$ has index $2$ in $\aut{\langle x\rangle}$, the $2'$-part $\tau_4$ must lie in $\bg G x$. Hence we have $\aut{\langle x\rangle}=\bg G x\langle\tau_6\rangle$ and, again by Lemma~\ref{easy}, $x$ is also $6$-semi-rational. Now, $6$ has order $2$ modulo $7$, thus in particular $6^4\equiv 1$ modulo $7$, as claimed. With a similar counting we see that, for $p=11$ or $p=13$, a generator of $\aut{\langle x\rangle}$ (namely $\tau_2$ in both cases) has a $2$-part $\tau_r$ (with $r=10$ and $r=15$, respectively) such that $r^4\equiv 1$ modulo $p$.

On the other hand, for $p=17$, the order of any generator of $\aut{\langle x\rangle}$ is obviously $2^4$ and the above claim fails. However, we will see in the Main Theorem (actually, in Proposition~\ref{tableirreducible}) that $p=17$ does not show up, and the conclusion will be that every quadratic rational Frobenius group is indeed $r$-semi-rational for a suitable integer $r$ with $r^4\equiv 1$ modulo the exponent of $G$.
\end{rem}

In view of Proposition~\ref{prop5.5}, it will be useful to generalize some results of \cite{bac} to the case of uniformly semi-rational groups. The following is Definition~4.3 of \cite{bac}.

\begin{definition}
 Let $H$ be a group, $p$ a prime, and $K$ an $\F_p H$-module such that the corresponding semidirect product $G=K\rtimes H$ is a Frobenius group. Let $\alpha: H\to{\rm{GL}}(K)$ be a representation arising from the module $K$. For $x\in K$, we denote by $Z_G(x)$ the subgroup of $\aut{\langle x \rangle}$ consisting of the automorphisms induced by elements of $\alpha(H)\cap \zent{{\rm{GL}}(K)}$.
\end{definition}

As observed in \cite{bac}, $Z_G(x)$ is the same for every non-trivial $x\in K$ and, if $K$ is absolutely irreducible as an $\F_pH$-module, then $Z_G(x)$ consists of the automorphisms induced by elements of $\alpha(\zent H )$. Note also that $Z_G(x)$ is contained in $\bg G x$.

To conclude this section, we prove a generalization of \cite[Lemma 4.4]{bac} to uniformly semi-rational groups that will be crucial for the Main Theorem.
As mentioned in the Introduction, given a Frobenius group $G$ with kernel $K$ and complement $H$, we denote by $G_n=K^n\rtimes H$ the Frobenius group given by the diagonal action of $H$ on the direct product of $n$ copies of $K$; note that, as easily seen, if $G_n$ is $r$-semi-rational for some integer $n>1$ then $G$ is $r$-semi-rational as well. We will also make use of the notation introduced in Lemma~\ref{easy}.

\begin{lemma}\label{lemma5.4bis}
    Let $H$ be a group, $p$ an odd prime, and $K$ an $\F_pH$-module such that the corresponding semidirect product $G=K\rtimes H$ is an r-semi-rational Frobenius group. Also, let $x$ be a non-trivial element of $K$ and let $n$ be an integer larger than $1$. Then the following properties are equivalent.
\begin{enumerate}
\item $G_n$ is r-semi-rational.
\item $[\aut{\langle x\rangle}: Z_G(x)]\leq 2$ and $Z_G(x)\langle \tau_r \rangle=\aut{\langle x \rangle}$.
\end{enumerate}
\end{lemma}

\begin{proof}
Assume property (1). Since $x$ is $r$-semi-rational in $G$, Lemma~\ref{easy} yields $[\aut{\langle x \rangle}:\bg G x]\leq 2$  and $\aut{\langle x \rangle}=\bg G x\langle \tau_r \rangle$; 
moreover, as $|x|=p$, the group $\aut{\langle x \rangle}$ is cyclic and $\bg G x$ is generated by an element $\tau_u\in \aut{\langle x \rangle}$ such that the order of $u$ modulo $p$ is divisible by $\frac{p-1}{2}$. Take $h\in H$ such that $x^h=x^u$, and consider an element $v=(x,y,y,\ldots,y)\in K^n$ where $y\in K$ is arbitrary: this element is semi-rational in $G_n$ as well, so the same argument shows that there exists $l\in H$ such that $v^l=v^u$. Now we get $(x^l,y^l,y^l,\ldots,y^l)=(x^u,y^u,y^u,\ldots,y^u)$, thus $x^l=x^u=x^h$ and, being the action of $H$ on $K$ fixed-point free, we have $l=h$; as a consequence, we also have $y^h=y^u$ for every $y\in K$, meaning that $\tau_u$ lies in $Z_G(x)$. It follows that $\langle \tau_u\rangle=\bg G x\subseteq Z_G(x)$ (so, in fact, equality holds), hence $[\aut{\langle x\rangle}: Z_G(x)]\leq 2$ and $Z_G(x)\langle \tau_r \rangle=\bg G x\langle \tau_r \rangle=\aut{\langle x \rangle}$, as wanted.

Conversely, assume that (2) holds. Then there exist an integer $u$ coprime to $p$ and an element $h\in H$ such that the order of $u$ modulo $p$ is divisible by $\frac{p-1}{2}$, and $y^h=y^u$ for every $y\in K$. Now, if $v=(x_1,\ldots,x_n)$ is any non-trivial element of $K^n$, we have $v^h=(x_1^h,\ldots,x_n^h)=(x_1^u,\ldots,x_n^u)=v^u$: in other words, $\tau_u\in\aut{\langle v\rangle}$ lies in $\bg {G_n} v$ and so $[\aut{\langle v\rangle}:\bg {G_n}v]\leq 2$. Moreover, if $\tau_r$ lies in $Z_G(x)$, i.e., if  $Z_G(x)=\aut{\langle x\rangle}$, then we get the following property: for every integer $j$ coprime to $p$, there exists $t\in H$ with $y^t=y^j$ for all $y\in K$, and it easily follows that $v$ is rational in $G_n$. On the other hand, if $\tau_r\not\in Z_G(x)$, then the order of $r$ modulo $p$ is not a divisor of $\frac{p-1}{2}$; therefore we get $\aut{\langle v\rangle}=\bg{G_n} v\langle\tau_r\rangle$, and Lemma~\ref{easy} yields that $v$ is $r$-semi-rational in $G_n$ also in this case. We conclude that every element of $K^n$ is $r$-semirational in $G_n$; since $H\cong G/K$ is an $r$-semi-rational group, $G_n=K^n\rtimes H$ is $r$-semi-rational, as wanted.  
\end{proof}

\begin{rem}\label{Gn}
Observe that, in the situation of Lemma~\ref{lemma5.4bis}, we have that $G_n$ is uniformly semi-rational for every $n\geq 2$ if and only if $G_2$ is uniformly semi-rational. If this is the case, then the semi-rationality $\sg_{G_n}$ of $G_n$ (which is easily seen to be a subset of $\sg_G$) is the same for all $n\geq 2$, and it is the set  of all the elements $r\in\sg_G$ such that, for a non-trivial $x\in K$ (hence for every $x\in K$), we have $Z_G(x)\langle \tau_r \rangle=\aut{\langle x \rangle}$.
\end{rem}
    
As we will see in the Main Theorem, if the complement $H$ of a quadratic rational Frobenius group has even order, then the Frobenius kernel $K$ is an elementary abelian $p$-group for a suitable (odd) prime~$p$. Thus, the situation analyzed in this section will turn out to cover a ``large" portion of the relevant cases.

\section{A proof of the Main Theorem}

In the previous sections we gathered the theoretical set-up that will be relevant for the Main Theorem. We are then ready to present a proof of it, and we will first discuss the classification of quadratic rational Frobenius groups having a complement of even order. This will be achieved with the four results whose statements are presented next. 

As we will observe in Proposition~\ref{tableirreducible}, the group $C_4$ has two non-isomorphic $1$-dimensional modules over $\F_5$ whose corresponding semidirect products are (isomorphic) quadratic rational Frobenius groups. Before stating the theorems we note that, in part (2) of Theorem~\ref{4}, the action of $C_4$ on each direct factor of $C_5^a$ corresponds to one of these module structures, whereas the action of $C_4$ on each direct factor of $C_5^b$ corresponds to the other. We also remark that the groups as in part (2) of Theorem~\ref{4} are the only groups, among those appearing in the conclusions of the four theorems, which are not uniquely determined (up to isomorphism) by the structure of their Frobenius kernel and complement.

\begin{theorem}\cite[Main Theorem]{qfrob}. \label{1}
    Let $G$ be a rational Frobenius group. Then $G$ is isomorphic to one of the following groups.
    \begin{enumerate}
        \item $C_3^n\rtimes C_2$.
        \item $C_3^{2n}\rtimes Q_8$.
        \item $C_5^2\rtimes Q_8$.
    \end{enumerate}
\end{theorem}

\begin{theorem} \cite[Theorem 1.3]{bac}.\label{2}
    Let $G$ be a non-rational, inverse semi-rational Frobenius group with complement of even order. Then $G$ is isomorphic to one of the following groups. 
        \begin{enumerate}
                \item $C_3^{2n}\rtimes C_4$, with $\sg_G=-1\cdot \langle 5\rangle$.
                \item $C_5^n \rtimes C_4$, with $\sg_G=-1\cdot \langle 13\rangle$.
                \item $C_7^n \rtimes C_6$, with $\sg_G=-1\cdot \langle 19\rangle$.
                \item $C_5^2\rtimes (C_3\rtimes C_4)$, with $\sg_G=-1\cdot \langle 17,41\rangle$.
                \item $C_5^2 \rtimes {\rm {SL}}_2(3)$, with $\sg_G=-1\cdot \langle 7,19\rangle$.
                \item $C_7^2\rtimes {\rm {SL}}_2(3)$, with $\sg_G=-1\cdot \langle 13, 19\rangle$.
                \item $C_7^{2n} \rtimes (C_3\times Q_8)$, with $\sg_G=-1\cdot \langle 19,43\rangle$.
            \end{enumerate} 
\end{theorem}

\begin{theorem}\label{3}
    Let $G$ be a quadratic rational Frobenius group which is not inverse semi-rational. Assume that the Frobenius complement $H$ has even order. Then  $G$ is uniformly semi-rational. In particular, if $G$ is $r$-semi-rational for some $r$ whose order modulo the exponent of $G$ is $2$, then $G$ is isomorphic to one of the following groups.
    \begin{enumerate}
            \item $C_3^{4n} \rtimes Q_{16}$, with $\sg_G=5\cdot \langle 7,23\rangle$.
            \item $C_7^2\rtimes ({\rm {SL}}_2(3).C_2)$, with  $\sg_G=5\cdot \langle 73, 113, 127\rangle$.
    \end{enumerate}
    
\end{theorem}

\begin{theorem}\label{4}

Let $G$ be a quadratic rational Frobenius group which is not any of the groups in Theorem~$\ref{1}$, Theorem~$\ref{2}$, or Theorem~$\ref{3}$. Assume that the Frobenius complement $H$ has even order. Then  
$G$ is $r$-semi-rational for a suitable $r$ whose order modulo the exponent of $G$ is $4$, and $G$ is isomorphic to one of the following groups.
    \begin{enumerate}
            \item  $C_5^n\rtimes C_2$, with $\sg_G=3\cdot \langle -1\rangle$.
            \item $(C_5^a\times C_5^b)\rtimes C_4$, with $\sg_G=3\cdot \langle 9 \rangle$.
            \item $C_5^{2n}\rtimes C_6$, with $\sg_G=-7\cdot \langle 19\rangle$.
            \item $C_{13}^n\rtimes C_6$, with $\sg_G=-7\cdot \langle 49\rangle$.
            \item $C_5^{2n}\rtimes Q_8$, with $n>1$ and $\sg_G=3\cdot \langle 9,11 \rangle$.
            \item $C_5^{2n}\rtimes (C_3\rtimes C_4)$, with $n>1$ and $\sg_G=7\cdot \langle 29, 41 \rangle$.
            \item $C_5^{4n}\rtimes Q_{16}$, with $\sg_G=3\cdot \langle 31,9\rangle$.
            \item $C_3^{4n}\rtimes H_1$, with $\sg_G=7\cdot \langle 41, 49 \rangle$.
            \item $C_5^{2n}\rtimes {\rm {SL}}_2(3)$, with $n>1$ and $\sg_G=-7\cdot \langle 19, 49 \rangle$.
            \item $C_5^{4n}\rtimes (C_3\rtimes Q_8)$, with $\sg_G=7\cdot \langle 11,49\rangle$.
            \item $C_5^{4n}\rtimes (C_3\times Q_8)$, with $\sg_G=-7\cdot\langle 19, -11\rangle$.
            \item $C_{13}^{2n}\rtimes (C_3\times Q_{8})$, with $\sg_G=-7\cdot \langle 49,79 \rangle$.
            \item $C_5^{4n}\rtimes H_2$, with $\sg_G=-13\cdot \langle 31, 49 \rangle$.
            \item $C_5^{4n}\rtimes ({\rm {SL}}_2(3).C_2)$, with $\sg_G=-13\cdot \langle 31,41,49 \rangle$.
            \item $C_{11}^2\rtimes {\rm {SL}}_2(5)$, with $\sg_G=-7\cdot \langle 541, 529, 221, 331 \rangle$.
    \end{enumerate}
 \end{theorem}
 
 The following proposition, based on computation with \texttt{GAP}, will be our starting point for the proof.
 
 \begin{proposition}\label{tableirreducible} Let $G$ be a Frobenius group with kernel $K$ and complement $H$ of even order. Assume that $K$ is a minimal normal subgroup of $G$. Then $G$ is quadratic rational if and only if its structure is one of those displayed in Table~$\ref{irreducibleK}$.
 \end{proposition} 
 
\begin{proof} In Proposition~\ref{prop4.6} we derived some restrictions on the structure of a quadratic rational Frobenius group $G$ whose complement has even order (see Table~\ref{possibleK}). In view of that and with the help of \texttt{GAP}, it is possible to derive the exact list of all the Frobenius groups of this kind with the additional property that the kernel $K$ is a minimal normal subgroup of $G$ (i.e., it is irreducible as an $H$-module). More precisely, we use the function ``\texttt{pIrreducibleSemidirectFrobeniusSemirational}" that can be found in \cite{codes}: this takes a group $H$ and a prime $p$ as an input, and it returns the list of all the Frobenius groups $G=K\rtimes H$ such that $K$ is an irreducible $\F_p H$-module and $G$ is quadratic rational, checking whether the relevant Frobenius actions have the $\frac{p-1}{2}$-eigenvalue property. 

Note that, in view of Proposition~\ref{prop5.5}, Table~\ref{irreducibleK} also provides the exact list of all the \emph{semi-rational} Frobenius groups whose complement has even order and whose kernel is minimal normal.
 
 In Table~\ref{irreducibleK}, when an isomorphism type of $K$ is marked with $(*)$, we mean that $K$ has two non-isomorphic structures as an $H$-module; however, in all these cases, the isomorphism type of $G=K\rtimes H$ is uniquely determined. The last column of Table~\ref{irreducibleK} also provides the \texttt{GAP} identifiers of the relevant groups when they exist (otherwise we write n.a., for ``not available").
\end{proof}

 \begin{table}[h]
    
   \begin{tabular}{|c|c|c|c|}
    \hline
    $|H|$ & $H$ & $K$ & ${\texttt{IdGroup(G)}}$\\
    \hline
    $2$& $C_2$ & $C_3,\quad C_5$ & ${\texttt{[6,1],\quad[10,1]}}$\\  
    $4$& $C_4$ & $C_3^2,\quad C_5\; (*)$ & ${\texttt{[36,9],\quad[20,3]}}$ \\ 
    $6$ & $C_6$ & $C_5^2,\quad C_7\;(*),\quad C_{13}\; (*)$ & ${\texttt{[150,6],\quad[42,1],\quad [78,1]}}$\\
    $8$ & $Q_8$ & $C_3^2,\quad C_5^2$ & ${\texttt{[72,41],\quad[200,44]}}$\\
     $12$ & $C_3\rtimes C_4$ & $C_5^2$ & ${\texttt{[300,23]}}$\\
    $16$ & $Q_{16}$ & $C_3^4,\quad C_5^4$ & ${\texttt{[1296,3523],\quad n.a.}}$\\
    $20$ & $H_1$ & $C_3^4$ & ${\texttt{[1620,419]}}$\\
    $24$ & ${\rm {SL}}_2(3)$ & $C_5^2,\quad C_7^2$ & ${\texttt{[600,150],\quad[1176,215]}}$\\
    $24$ &$ C_3\rtimes Q_8$ & $C_5^4$ & ${\texttt {n.a.}}$\\
    $24$ &$C_3\times Q_8$ & $C_5^4,\quad C_7^2\;(*),\quad C_{13}^2\;(*)$ & ${\texttt{n.a.,\quad [1176,218],\quad n.a.}}$ \\
    $48$ & $H_2$ & $C_5^4$ & ${\texttt {n.a.}}$\\
    $48$ & ${\rm {SL}}_2(3).C_2$ & $C_5^4,\quad C_7^2\;(*)$ & ${\texttt {n.a.,\quad n.a.}}$\\
    $120$ & ${\rm {SL}}_2(5)$ & $C_{11}^2\; (*)$ & ${\texttt {n.a.}}$\\
    \hline
    \end{tabular}
    \vspace{0.3cm}
    
    \caption{Quadratic rational Frobenius groups $G=K\rtimes H$ with $|H|$ even and $K$ minimal normal.}
        \label{irreducibleK}
    \end{table}

\begin{proof}[Proof of Theorems~$\ref{1}-\ref{4}$] By our assumptions, $G$ is a quadratic rational Frobenius group whose complement $H$ has even order: in this situation, Proposition~\ref{prop4.5} yields that $K$ is an abelian group of (odd) square-free exponent and, for every prime $p\in\pi(K)$, the Sylow $p$-subgroup of $K$ is a normal subgroup of $G$ that can be viewed as an $\F_pH$-module. Now, by Maschke's Theorem, every Sylow subgroup of $K$ is in fact a direct product of minimal normal subgroups of $G$ that can be viewed as irreducible $\F_pH$-modules. Note that, if $V$ is one of these minimal normal subgroups, then $VH$ is isomorphic to a factor group of $G$, hence it is a quadratic rational Frobenius group as well and it must show up in Table~\ref{irreducibleK}. 

Assume first that $K$ is not a group of prime-power order. Taking into account Table~\ref{possibleK} and the paragraph above, the possible isomorphism types of $H$ in this situation are $C_4$, $Q_8$, $Q_{16}$ (with $\pi(K)=\{3,5\}$ for all three cases) or $C_3\times Q_8$ (with \(\pi(K)=\{5,7\}\)). 

As regards the case $\pi(K)=\{3,5\}$, $K$ must have two subgroups $V_1\cong C_3^n$ and $V_2\cong C_5^m$, both minimal normal in $G$, such that $(n,m)$ is $(2,1)$, $(2,2)$ or $(4,4)$ depending on whether $H\cong C_4$, $H\cong Q_8$, or $H\cong Q_{16}$ respectively; moreover, $V_1H$ and $V_2H$ are quadratic rational Frobenius groups as in Table~\ref{irreducibleK}. Set $W=V_1V_2$ and consider the Frobenius group $WH$, which is also quadratic rational (because it is isomorphic to a factor group of $G$). Take $\lambda\in\irr W$ with $\Q(\lambda)=\Q_{15}$ (and $|\lambda|=15$): the quadratic rationality of $WH$ implies that $I^*_\lambda/I_\lambda$ is isomorphic to a subgroup of index at most $2$ of $\Gal{(\Q_{15}/\Q)}$. On the other hand, $I^*_\lambda/I_\lambda$ is also isomorphic to a subgroup of $H$, hence it must be cyclic of order $4$; but now $I^*_\lambda/I_\lambda$, which stabilizes the subgroup $\langle\lambda\rangle$ of $\irr W$, cannot act fixed-point freely on it because $15\not\equiv 1$ (mod $4$). This is a contradiction, because the action of $H$ on $\irr W$ is fixed-point free, and we conclude that $G$ cannot have a kernel $K$ with $\pi(K)=\{3,5\}$.

An entirely similar argument shows that also the assumption $\pi(K)=\{5,7\}$ (with $H\cong C_3\times Q_8$) leads to a contradiction, thus our conclusion so far is that $K$ is an elementary abelian $p$-group for a suitable prime $p$. In particular, by Proposition~\ref{prop5.5}, $G$ is uniformly semi-rational. We also note that, at this stage, the claim that $G$ is $r$-semi-rational for a suitable $r$ whose order modulo the exponent of $G$ divides $4$ follows by Remark~\ref{USR} (and by the fact that the prime $17$ does not show up as a divisor of $|K|$ in Table~\ref{irreducibleK}).

Next, we consider the situation when the $p$-group $K$ decomposes into a direct product of minimal normal subgroups of $G$ that are not isomorphic as $\F_p H$-modules (i.e., when the $\F_pH$-module $K$ is not homogeneous); this can happen when the isomorphism type of $H$ corresponds to a row of Table~2 containing the symbol $(*)$.

Let us start with the case $H=\langle h\rangle\cong C_4$, so, we consider a Frobenius group  $G$ of the kind $(C_5^a\times C_5^b)\rtimes C_4$ where $x^h=x^2$ for every $x\in C_5^a$, and $y^h=y^3$ for every $y\in C_5^b$. Take a non-trivial element $v=(x,y)\in C_5^a\times C_5^b$, and observe that $v^{h^2}=(x^4,y^4)=v^4$; it follows that $\tau_4\in\aut{\langle v\rangle}$ lies in $\bg G v$, and it has order $2$. Hence we get $[\aut{\langle v\rangle}:\bg G v]=2$ and $\aut{\langle v\rangle}=\bg G v\langle \tau_3\rangle$, and it is easily seen (also by Lemma~\ref{easy}) that we are in case (5) of Theorem~\ref{4}. 

Let now $H=\langle h\rangle\cong C_6$, and consider a Frobenius group $G$ of the kind $(C_{7}^a\times C_{7}^b)\rtimes C_6$ where $x^h=x^3$ for every $x\in C_{7}^a$, and $y^h=y^5$ for every $y\in C_{7}^b$. Taking $v=(x,y)\in C_{7}^a\times C_{7}^b$ with non-trivial $x$ and $y$, it is easy to see that $\bg G v=\langle\tau_{6}\rangle$ has index $3$ in $\aut{\langle v\rangle}$. Therefore, $G$ is not semi-rational and this case can be discarded.

What we did in the situation of the paragraph above was to check that the $\F_p H$-module $K$ is not semi-rational in $G$ (i.e., it does not have the $\frac{p-1}{2}$-eigenvalue property), producing an element of $K$ that is not semi-rational. With the same technique, also with the help of the \texttt{GAP} function ``\texttt{checkingRemainingCases}" in \cite{codes}, it is possible to discard all the remaining cases in which $K$ is non-homogeneous.

It remains to consider each group $G$ arising from Table~\ref{irreducibleK} and (taking into account Proposition~\ref{prop5.5}) determine the integers $r$ for which $G$ is $r$-semi-rational: this can be done via the function ``\texttt{USRInt}" of~\cite{codes}. Finally we need to check whether or not, for $n>1$ and for these integers $r$, the group $G_n=K^n\rtimes H$ is $r$-semi-rational: to this end we use the characterization obtained in Lemma~\ref{lemma5.4bis} and, with the help of the function ``\texttt{extensionsSG}" in \cite{codes}, we test condition (2) of that lemma on the group $G$.  The result is what we stated in Theorems~\ref{1}-\ref{4}.
\end{proof}
    
Note that ${\rm{SL}}_2(5).C_2$ actually does not occur as the complement of a quadratic rational Frobenius group; on the other hand, all the other groups listed in Theorem~\ref{qrcompl} and  Proposition~\ref{nonsolvablefrobcompl} do occur (as we see from the results above, or also from Proposition~\ref{tableirreducible}, with the observation that $C_7\rtimes C_3$ is inverse semi-rational). Therefore we have also completed the proof of Theorem~A.  

\begin{rem}\label{semiratcompleted}
Assume that $G$ is a \emph{semi-rational} Frobenius group whose complement $H$ has even order. Then the kernel $K$ of $G$ (is abelian and) has square-free exponent by Lemma~2.5 of \cite{semiratfrobenius}. If $V$ is a minimal normal subgroup of $G$ contained in $K$ then, as in the first papagraph of the proof above, we see that $VH$ is a Frobenius group that must show up in Table~\ref{irreducibleK} (recall also Proposition~\ref{tableirreducible}). Now, a slight variation of the argument in the second, third and fourth paragraphs of the same proof yields that $K$ must be a group of prime-power order, hence $G$ turns out to be quadratic rational by Proposition~\ref{prop5.5}.

Since we will see next (in Proposition~\ref{ultima}) that a semi-rational Frobenius group whose complement has odd order is inverse semi-rational, hence also quadratic rational, we conclude that the class of semi-rational Frobenius groups \emph{coincides} with that of quadratic rational Frobenius groups. Thus, the Main Theorem of this paper can be also regarded as a completion of the analysis carried out in \cite{semiratfrobenius}.
\end{rem}
   
\smallskip
To finish the proof of the Main Theorem, it remains to treat the case of quadratic rational Frobenius groups whose complement is isomorphic to $C_3$. In this case the Frobenius kernel can be in principle of even order, so we need to consider quadratic rational $2$-groups.

\begin{proposition}\label{twogroups}
        Let $G$ be a $2$-group of nilpotency class at most $2$, and assume that $G$ is either quadratic rational or semi-rational. Then the following properties hold.
        \begin{enumerate}
            \item The abelian groups $G/\zent G$ and $\zent G$ have an exponent that divides $4$.
            \item $G$ is either rational or inverse semi-rational with semi-rationality $\sg_G=-1\cdot\langle 5 \rangle$.
        \end{enumerate}
    \end{proposition}
    \begin{proof}
We start by proving claim (1). Since $G/\zent G$ is an abelian $2$-group that is quadratic rational or semi-rational, it has an exponent that divides $4$ by Proposition~\ref{qrfacts} and \cite[Lemma~5]{CD}, respectively. As regards $\zent G$, let us consider the case when $G$ is semi-rational: for $x\in\zent G$ we have on one hand $|\bg G x|=1$, and on the other hand $|\aut{\langle x\rangle}:\bg G x|\leq 2$ by Lemma~\ref{easy}. We conclude that $\phi(|x|)\leq 2$, hence $x^4=1$. Assume now that $G$ is quadratic rational, and let $\lambda\in\irr{\zent G}$ be such that $\Q(\lambda)=\Q_m$ where $m$ denotes the exponent of $\zent G$. Observe that $G= I_\lambda$, so $I^*_\lambda/I_\lambda=1$; thus, By Lemma~\ref{lemmant}, we have that $\Q(\lambda^G)=\Q(\lambda)$. We conclude that $|\Q_m:\Q|=|\Q(\lambda):\Q|=|\Q(\lambda^G):\Q|\leq 2$ (recall Proposition~\ref{normalinqr}) and therefore $m$ divides $4$, as wanted.  

Next, we prove (2). If $G$ is abelian then we are done by (1), so we can assume that $G$ is non-abelian and we consider an element $x\in G\setminus \zent G$. Since $G'\subseteq\zent G$, we have $[x,g]\in \zent G$ for every $g\in G$, and we get $[x,g][x,h]=[x,gh]$ for every $g,h\in G$. In particular, the set $W=\{[x,g]\mid g\in G\}$ is a non-trivial central subgroup of $G$ and $G/W$ is a quadratic rational or a semi-rational group as well (thus property (1) holds for $G/W$). Observe that, for every $g\in G$, we have $(xW)^g=x^gW=x[x,g]W=xW$, i.e., $xW$ lies in $\zent {G/W}$. But, as the exponent of $\zent{G/W}$ divides $4$, we get $x^4\in W$; in other words, $x^5$ is conjugate to $x$ in $G$.  The same conclusion holds also for every element of $\zent G$, considering that the exponent of $\zent G$ is a divisor of $4$; therefore, denoting by $n$ the exponent of $G$ (which divides $16$), we conclude that $\langle 5 \rangle\subseteq \rg_G\cong \Gal(\Q_{n}/\Q(G))$. But, as the index of $\langle 5\rangle$ in $\mathcal{U}(\Z/n\Z)$ is at most $2$, this yields $|\Q(G): \Q|\leq 2$, hence $\Q(\chi)\subseteq \Q(i)$ for every $\chi\in\irr G$. The desired conclusion now follows from~\cite[Proposition 2.2]{bac}. \end{proof}
    
The proposition above shows that, for a $2$-group of nilpotency class at most $2$, the properties of being quadratic rational, semi-rational or inverse semi-rational are all equivalent. However, it is possible to find examples of $2$-groups $G$ having nilpotency class $3$ and such that:
\begin{enumerate}
\item $G$ is quadratic rational and semi-rational, not inverse semi-rational (\texttt{SmallGroup(128,417)}).
\item $G$ is quadratic-rational, not semi-rational (\texttt{SmallGroup(32,9)}).
\item $G$ is semi-rational, not quadratic rational (\texttt{SmallGroup(32,42)}).
\item $G$ is neither quadratic rational nor semi-rational (\texttt{SmallGroup(64,15)}).
\end{enumerate}

We are ready to conclude the proof of the Main Theorem. In view of Theorem~1.3 of \cite{bac}, it only remains to prove the following result.

\begin{proposition}\label{ultima}
Let $G$ be a Frobenius group whose complement has odd order. If $G$ is either quadratic rational or semi-rational, then $G$ is inverse semi-rational.
\end{proposition}

\begin{proof} Observe first that, by Theorem~\ref{qrcompl} and \cite[Theorem~1.2]{semiratfrobenius}, we have $H\cong C_3$. Thus, denoting by $S$ the Sylow $2$-subgroup of $K$, we have that $S$ has a fixed-point free automorphism of order $3$, and therefore the nilpotency class of $S$ is at most $2$ by \cite[Theorem 3]{higman}. We claim that $S$ is either quadratic rational or semi-rational, and this will imply that $S$ is inverse semi-rational by an application of Proposition~\ref{twogroups}.

In fact, set $U=\oh{2'} K$ and consider the subgroup $SH\cong G/U$ of $G$. If $G$ is quadratic-rational, then $SH$ is a quadratic rational Frobenius group as well; therefore, for $\lambda\in\irr S$, we have $|\Q(\lambda^{SH}):\Q|\leq 2$. Now, $I^*_\lambda/I_\lambda$ is isomorphic to a subgroup of $H$, but also to the $2$-group $\Gal{(\Q(\lambda)/\Q(\lambda^{SH}))}$; as a consequence, $I^*_\lambda/I_\lambda$ is trivial and so $\Q(\lambda)=\Q(\lambda^{SH})$ is an extension of $\Q$ of degree at most $2$. In other words, $S$ is quadratic rational. On the other hand assume that $G$ is semi-rational, so that $SH$ is a semi-rational Frobenius group; for $x\in S$ we have that $\norm {SH}{\langle x\rangle}$ lies in $S$ (in fact, any $3$-element in  $\norm {SH}{\langle x\rangle}$ would actually centralize $x$, but $3$ does not divide $|\cent G x|$ because the action of $H$ on $S$ is fixed-point free), hence $\bg {SH} x=\bg S x$. It follows that, $x$ being semi-rational in $SH$, $x$ is semi-rational in $S$ as well, and the claim of the paragraph above is proved.

Our conclusion so far is that $S$ is inverse semi-rational, hence, in particular, every $2$-element of $G$ is inverse semi-rational in $G$. Next, consider the quadratic rational or semi-rational group $G/S\cong UH$. Since this group has odd order, it is inverse semi-rational by Proposition~\ref{oddequiv}; in particular every element of $U$ is inverse semi-rational in $G$. Moreover, by the discussion in the second paragraph of Remark~\ref{semiratorder}, $UH$ is a $\{3,7\}$-group (i.e., $U$ is a $7$-group). Since every $3$-element of $G$ is inverse semi-rational in $G$ as well as the $2$-elements and the $7$-elements, the desired conclusion follows if we prove that $G$ is either a $\{2,3\}$-group or a $\{3,7\}$-group. So, arguing by contradiction, assume $\pi(K)=\{2,7\}$ and take an element $x\in\zent K$ of order $14$. Then $x$ is semi-rational in $G$ (recall Lemma~\ref{semiratpel}), and this forces $\bg G x$ to have an element of order $3$. In other words, $\bg G x=G/K$ and so $\langle x\rangle$ is normal in $G$; this is clearly a contradiction, as it would imply that $H$ centralizes the non-trivial element $x^7\in K$. The proof is complete.
\end{proof}

\smallskip
{\bf Acknowledgment.} The authors would like to thank \'Angel del R\'io for some fruitful discussions on the subject of this work (in particular, for suggesting the concept of uniform semi-rationality), and for the kind hospitality offered to the second author at the University of Murcia. They also thank the referee for her/his careful reading and for some valuable comments that improved the exposition of the material in this paper.

\end{document}